\documentclass[11pt]{amsart}
\usepackage{amsthm,amssymb,mathrsfs,mathtools,empheq}

\usepackage[T1]{fontenc}

\usepackage[notref,notcite,final]{showkeys}
\mathtoolsset{showonlyrefs}
\usepackage{fullpage,color}
\allowdisplaybreaks

\newtheorem{mainthm}{Theorem}
\newtheorem{theorem}{Theorem}[section]
\newtheorem*{theorem*}{Theorem}
\newtheorem{corollary}[theorem]{Corollary}

\newtheorem{lemma}[theorem]{Lemma}
\newtheorem{proposition}[theorem]{Proposition}
\newtheorem*{proposition*}{Proposition}

\newtheorem*{conjecture*}{Conjecture}

\theoremstyle{definition}

\newtheorem{remark}[theorem]{Remark}

\numberwithin{equation}{section}

\def\bN {\mathbb{N}}

\def\bR {\mathbb{R}}

\def\bZ {\mathbb{Z}}

\def\cA {\mathcal{A}}

\def\cE {\mathcal{E}}
\def\cF {\mathcal{F}}

\def\cM {\mathcal{M}}
\def\cN {\mathcal{N}}

\def\cR {\mathcal{R}}

\def\cX {\mathcal{X}}
\def\cY {\mathcal{Y}}
\def\cZ {\mathcal{Z}}

\def\scrL{\mathscr{L}}

\def\grad {{\nabla}}

\def\rstr {{\big |}}

\def\la {\langle}
\def\ra {\rangle}

\newcommand{\sto}[1]{\xrightarrow[#1]{}}

\newcommand{\tx}[1]{\mathrm{#1}}

\newcommand{\wt}[1]{\widetilde{#1}}
\newcommand{\bs}[1]{\boldsymbol{#1}}

\newcommand{\supp}{\operatorname{supp}}

\renewcommand{\ker}{\operatorname{ker}}

\newcommand{\Id}{\operatorname{Id}}

\newcommand{\eee}{\mathrm e}

\newcommand{\ud}{\mathrm{\,d}}
\newcommand{\vd}{\mathrm{d}}
\newcommand{\vD}{\mathrm{D}}
\newcommand{\dd}[1]{{\frac{\vd}{\vd{#1}}}}

\newcommand{\uln}[1]{{\underline{ #1 }}}
\newcommand{\lin}{_{\textsc{lin}}}

\newcommand{\yp}{{\cY}^+}
\newcommand{\ym}{{\cY}^-}
\newcommand{\ypl}{{\cY}^+_{\lambda}}
\newcommand{\yml}{{\cY}^-_{\lambda}}

\title{Nonexistence of radial two-bubbles with opposite signs for~the energy-critical wave equation}
\author{Jacek Jendrej}
\address{\'Ecole Polytechnique, CMLS, 91128 Palaiseau, France}
\email{jacek.jendrej@polytechnique.edu}

\begin{document}

\begin{abstract}
We consider the focusing energy-critical wave equation in space dimension $N \geq 3$ for radial data.
We study two-bubble solutions, that is solutions which behave as a superposition of two decoupled radial ground states
(called bubbles) asymptotically for large positive times.
We prove that in this case these two bubbles must have the same sign.
The main tool is a sharp coercivity property of the energy functional near the family of ground states.
\end{abstract}
\maketitle

\section{Introduction}
\label{sec:intro}

\subsection{Setting of the problem and the main result}
\label{ssec:intro-setting}
Let $N \geq 3$ be the dimension of the space. For $\bs u_0 = (u_0, \dot u_0) \in \cE := \dot H^1(\bR^N) \times L^2(\bR^N)$, define the \emph{energy functional}
\begin{equation*}
  E(\bs u_0) = \int\frac 12|\dot u_0|^2 + \frac 12|\grad u_0|^2 - F(u_0)\ud x,
\end{equation*}
where $F(u_0) := \frac{N-2}{2N}|u_0|^\frac{2N}{N-2}$. Note that $E(\bs u_0)$ is well-defined due to the Sobolev Embedding Theorem.
The differential of $E$ is $\vD E(\bs u_0) = (-\Delta u_0 - f(u_0), \dot u_0)$, where $f(u_0) = |u_0|^\frac{4}{N-2}u_0$.

We consider the Cauchy problem for the energy critical wave equation:
\begin{empheq}{equation}
  \label{eq:nlw}
  \bigg\{
    \begin{aligned}
      \partial_t \bs u(t) &= J\circ \vD E(\bs u(t)), \\
      \bs u(t_0) &= \bs u_0 \in \cE.
    \end{aligned}
    \tag{NLW}
  \end{empheq}
  Here, $J := \begin{pmatrix}0 & \Id \\ -\Id & 0 \end{pmatrix}$ is the natural symplectic structure. This equation is often written in the form
    \begin{equation*}
      \partial_{tt} u = \Delta u + f(u).
    \end{equation*}

    Equation \eqref{eq:nlw} is locally well-posed in the space $\cE$, see for example \cite{GSV92} and \cite{ShSt94} (the defocusing case),
    as well as a complete review of the Cauchy theory in \cite{KeMe08} (for $N \in \{3, 4, 5\}$) and \cite{BCLPZ13} (for $N \geq 6$).
    In particular, for any initial data $\bs u_0 \in \cE$ there exists a maximal time of existence $(T_-, T_+)$, $-\infty \leq T_- < t_0 < T_+ \leq +\infty$,
    and a unique solution $\bs u \in C((T_-, T_+); \cE)$. In addition, the energy $E$ is a conservation law.
    In this paper we always assume that the initial data is radially symmetric. This symmetry is preserved by the flow.

    For functions $v \in \dot H^1$, $\dot v \in L^2$, $\bs v = (v, \dot v)\in \cE$ and $\lambda > 0$, we denote
    \begin{equation*}
      v_\lambda(x) := \frac{1}{\lambda^{(N-2)/2}} v\big(\frac{x}{\lambda}\big), \qquad \dot v_\uln\lambda(x) := \frac{1}{\lambda^{N/2}} \dot v\big(\frac{x}{\lambda}\big),\qquad\bs v_\lambda(x) := \big(v_\lambda, \dot v_\uln\lambda\big).
    \end{equation*}

    A change of variables shows that
    \begin{equation*}
      E\big((\bs u_0)_\lambda\big) = E(\bs u_0).
    \end{equation*}
    Equation~\eqref{eq:nlw} is invariant under the same scaling: if $\bs u(t) = (u(t), \dot u(t))$ is a solution of \eqref{eq:nlw} and $\lambda > 0$, then
    $
    t \mapsto \bs u\big((t-t_0)/\lambda\big)_\lambda
    $ is also a solution
    with initial data $(\bs u_0)_\lambda$ at time $t = 0$.
    This is why equation~\eqref{eq:nlw} is called \emph{energy-critical}.

    A fundamental object in the study of \eqref{eq:nlw} is the family of stationary solutions $\bs u(t) \equiv \pm\bs W_\lambda = (\pm W_\lambda, 0)$, where
    \begin{equation*}
      W(x) = \Big(1 + \frac{|x|^2}{N(N-2)}\Big)^{-(N-2)/2}.
    \end{equation*}
    The functions $W_\lambda$ are called \emph{ground states}. 
    They are the only radially symmetric solutions and, up to translation, the only positive solutions
    of the critical elliptic problem
    \begin{equation}
      \label{eq:elliptic}
    -\Delta u - f(u) = 0.
  \end{equation}
  Note however that classification of nonradial solutions of \eqref{eq:elliptic} is an open problem
  (see \cite{delPino} for details).

  Recall that the \emph{Soliton Resolution Conjecture} predicts that a generic bounded (in a suitable sense)
  solution of a hamiltonian system asymptotically decomposes as a sum of decoupled solitons and a dispersion.
  This belief is based mainly on the analysis of completely integrable systems, for instance \cite{EckSch83}.
  The only complete classification of the dynamical behaviour of a non-integrable hamiltonian system is the result of
  Duyckaerts, Kenig and Merle \cite{DKM4}, which we recall here for the reader's convenience:
  \begin{mainthm}[\cite{DKM4}]
    \label{thm:DKM}
    Let $N = 3$ and let $\bs u(t): [t_0, T_+) \to \cE$ be a radial solution of \eqref{eq:nlw}. Then one of the following holds:
    \begin{itemize}
      \item \textbf{Type I blow-up:} $T_+<\infty$ and 
        \begin{equation}
          \label{BUp}
          \lim_{t\to T_+} \|\bs u(t)\|_\cE =+\infty. 
        \end{equation} 
      \item \textbf{Type II blow-up:} $T_+<\infty$ and there exist $\bs v_0 \in \cE$, an integer $n\in \bN\setminus\{0\}$, and for all $j\in \{1,\ldots,n\}$, a sign $\iota_j\in \{\pm 1\}$, and a positive function $\lambda_j(t)$ defined for $t$ close to $T_+$ such that
        \begin{gather}
          \label{hyp_lambda_bup}
          \lambda_1(t)\ll \lambda_2(t)\ll \ldots \ll\lambda_{n}(t)\ll T_+-t\text{ as }t\to T_+\\
          \label{expansion_u_bup}
          \lim_{t\to T_+}
          \big\|\bs u(t)-\big(\bs v_{0}+\sum_{j=1}^n\iota_j \bs W_{\lambda_j(t)}\big)\big\|_\cE=0.
        \end{gather}
      \item \textbf{Global solution:} $T_+=+\infty$ and there exist a solution $\bs v_{\lin}$ of the linear wave equation, an integer $n\in \bN$, and for all $j\in \{1,\ldots,n\}$, a sign $\iota_j\in \{\pm 1\}$, and a positive function $\lambda_j(t)$ defined for large $t$ such that
        \begin{gather}
          \label{hyp_lambda}
          \lambda_1(t)\ll \lambda_2(t)\ll \ldots \ll\lambda_{n}(t)\ll t\text{ as }t\to +\infty\\
          \label{expansion}
          \lim_{t\to+\infty}
          \big\|\bs u(t)-\big(\bs v_{\lin}(t)+\sum_{j=1}^n\iota_j \bs W_{\lambda_j(t)}\big)\big\|_\cE=0.
        \end{gather}
    \end{itemize}
    \qed
  \end{mainthm}
%
  Of special interest are the solutions which are bounded in $\cE$ and which exhibit \emph{no dispersion}
  (that is, $\bs v_0 = 0$ or $\bs v\lin = 0$) in one or both time directions.
  One of the consequences of the \emph{energy channel estimates} in \cite{DKM4}
  is that in the case $N = 3$ the only solutions without any dispersion in both time directions are the stationary states
  $\bs W_\lambda$. This is in contrast with the case of completely integrable systems.
  
  In the present paper we are interested in solutions with no dispersion in one time direction, say for positive times.
  According to Theorem~\ref{thm:DKM}, for $N = 3$ such a solution has to behave asymptotically as a decoupled superposition
  of stationary states. Such solutions are called (pure) \emph{multi-bubbles} (or $n$-bubbles, where $n$ is the number of bubbles).
  By conservation of energy, if $\bs u(t)$ is an $n$-bubble, then
  \begin{equation}
    E(\bs u(t)) = n E(\bs W).
  \end{equation}
  The case $n = 1$ in dimension $N \in \{3, 4, 5\}$ was treated by Duyckaerts and Merle \cite{DM08},
  who obtained a complete classification of solutions of \eqref{eq:nlw} at energy level $E(\bs u(t)) = E(\bs W)$.
  In particular, the only $1$-bubbles are $\bs W_\lambda$, $\bs W_\lambda^-$ and $\bs W_\lambda^+$,
  where $\bs W^-$ and $\bs W^+$ are some special solutions converging exponentially to $\bs W$.
  The authors solve also the \emph{reconnection problem} by showing that for negative times
  $\bs W^-$ scatters and $\bs W^+$ blows up in norm $\cE$ in finite time.

  Solutions of \eqref{eq:nlw} satisfying \eqref{expansion_u_bup} or \eqref{expansion} with $\bs v_0 \neq 0$ or $\bs v\lin \neq 0$
  can exhibit non-trivial dynamical behaviour, see the results of Krieger, Nakanishi and Schlag \cite{KrScTa09},
  Hillairet and Rapha\"el \cite{HiRa12}, Donninger and Krieger \cite{DonKr13}, Donninger, Huang, Krieger and Schlag \cite{DHKS14}
  and the author \cite{moi15p}.

  In the present paper we address the case $n = 2$, and more specifically the situation when the two bubbles have opposite signs.
\begin{mainthm}
  \label{thm:deux-bulles}
  Let $N \geq 3$. There exists no radial solutions $\bs u: [t_0, T_+) \to \cE$ of \eqref{eq:nlw} such that
  \begin{equation}
    \label{eq:deux-bulles}
    \lim_{t \to T_+}\|\bs u(t) - \bs W_{\lambda_1(t)} + \bs W_{\lambda_2(t)}\|_\cE = 0
  \end{equation}
  and
  \begin{itemize}
    \item in the case $T_+ < +\infty$, $\lambda_1(t)\ll \lambda_2(t)\ll T_+-t\text{ as }t\to T_+$,
    \item in the case $T_+ = +\infty$, $\lambda_1(t)\ll \lambda_2(t)\ll t\text{ as }t\to +\infty$.
  \end{itemize}
\end{mainthm}
\begin{remark}
  There exist no examples of solutions of \eqref{eq:nlw} such that expansion \eqref{expansion_u_bup} or \eqref{expansion}
  holds with $n > 1$ (with or without dispersion). Note however that spatially decoupled non-radial multi-bubbles
  were recently constructed by Martel and Merle \cite{MaMe15p} using the Lorentz transform.
  In their setting, the choice of signs seems to have little importance.
  
  On the other side, Theorem~\ref{thm:deux-bulles} is, to my knowledge, the only result proving non-existence
  of solutions of type multi-bubble for \eqref{eq:nlw} in some specific cases. Existence of pure two-bubbles with the same sign is an open problem.
\end{remark}
\begin{remark}
In the case of corotational wave maps, where existence of pure two-bubbles with the same orientation
is easily excluded for variational reasons. Our proof might be seen as an adaptation of this argument
to the case where the energy functional is not coercive.

Note that for corotational wave maps existence of pure two-bubbles with opposite orientations is an open problem,
related to the \emph{threshold conjecture} for degree $0$ equivariant wave maps, see C\^ote, Kenig, Lawrie and Schlag \cite{CKLS15}.
\end{remark}
\begin{remark}
  For the corresponding slightly sub-critical elliptic problem positive multi-bubbles cannot form,
  whereas multi-bubbles with alternating signs exist, see Li \cite{Li95}, Pistoia and Weth \cite{Pistoia07}.
\end{remark}

\subsection{Outline of the proof}

  \paragraph{\textbf{Step 1.}} The linearization of \eqref{eq:nlw} around $\bs W_\lambda$ has a stable direction $\cY_\lambda^-$.
  We construct \emph{stable manifolds} $\bs U_\lambda^a$ which are forward invariant sets
  tangent to $\cY_\lambda^-$ at $\bs W_\lambda$. They have good regularity and decay properties.
  They allow to define the \emph{refined unstable mode} $\beta_\lambda^a \in \cE^*$ with the following crucial property.

Decompose any initial data close to the family of stationary states as $\bs u_0 = \bs U_\lambda^a + \bs g$,
with $\bs g$ satisfying natural orthogonality conditions by an appropriate choice of $\lambda$ and $a$. We have the alternative:
\begin{itemize}
  \item (Coercivity) $|\la\beta_\lambda^a, \bs g\ra| \lesssim \|\bs g\|_\cE^2$, which implies $E(\bs u_0) - E(\bs W) \gtrsim \|\bs g\|_\cE^2$,
  \item (Destabilization) $|\la \beta_\lambda^a, \bs g\ra| \gg \|\bs g\|_\cE^2$, which implies the exponential growth of $|\la \beta_\lambda^a, \bs g\ra|$.
\end{itemize}
In other words, $\beta_\lambda^a$ provides an explicit way of controlling how solutions which violate
the coercivity of energy leave a neighbourhood of the stationary states for positive times.
\paragraph{\textbf{Step 2.}}Let $\bs u(t): [t_0; T_+) \to \cE$ be a solution of \eqref{eq:nlw} which satisfies \eqref{eq:deux-bulles}.
  As already mentioned, this implies that
  \begin{equation}
    \label{eq:energy}
    E(\bs u) = 2E(\bs W).
  \end{equation}
 We decompose for any $t \in [t_0, T_+)$:
$$
\bs u(t) = \bs U_{\lambda_2(t)}^{a_2(t)} - \bs W_{\lambda_1(t)} + \bs g(t),\qquad \lambda_1(t) \ll \lambda_2(t),
$$
with $\bs g(t)$ satisfying natural orthogonality conditions (in fact we use a suitable localization of $\bs W_{\lambda_1(t)}$).
From the Taylor formula we obtain
\begin{equation}
  \label{eq:taylor}
  E(\bs u) = E(\bs U_{\lambda_2}^{a_2} - \bs W_{\lambda_1}) + \la \vD E(\bs U_{\lambda_2}^{a_2} - \bs W_{\lambda_1}, \bs g\ra
  + \frac 12 \la \vD^2 E(\bs U_{\lambda_2}^{a_2} - \bs W_{\lambda_1})\bs g, \bs g\ra + o(\|\bs g\|_\cE^2).
\end{equation}
An explicit key computation shows that
$$
E(\bs U_{\lambda_2}^{a_2} - \bs W_{\lambda_1}) - 2E(\bs W) \gtrsim (\lambda_1 / \lambda_2)^\frac{N-2}{2}.
$$
It is at this point that the sign condition is decisive.
\paragraph{\textbf{Step 3.}}
We prove that the assumption that $\bs u(t)$ stays close to a $2$-bubble implies that
$|\la\beta_{\lambda_2}^{a_2}, \bs g\ra| \lesssim \|\bs g\|_\cE^2 + (\lambda_1 / \lambda_2)^\frac{N-2}{2}$.
This allows to show that the second term in the expansion \eqref{eq:taylor} is $\ll \|\bs g\|_\cE^2 + (\lambda_1 / \lambda_2)^\frac{N-2}{2}$.

Finally, by an elementary analysis of the linear stable and unstable modes we can prove that, at least along a sequence of times,
the third term of the expansion \eqref{eq:taylor} is coercive, that is $\gtrsim \|\bs g\|_\cE^2$. Inserted in \eqref{eq:taylor},
this leads to $E(\bs u) > 2E(\bs W)$, contradicting \eqref{eq:energy}.

\subsection{Acknowledgements}
\label{ssec:intro-merci}
This paper has been prepared as a part of my \mbox{Ph.~\!D.} under supervision of Y.~Martel and F.~Merle.
I would like to express my gratitude for their constant support and many helpful discussions.
A part of this work has been realized while I was visiting the University of~Chicago.
I would like to thank C.~Kenig and the Departement of~Mathematics for their hospitality.
I am grateful to R.~C\^ote, T.~Duyckaerts and K.~Nakanishi for useful remarks.
The author has been supported by the ERC~grant $291214$ BLOWDISOL.

\subsection{Notation}
\label{ssec:intro-notation}
We introduce the infinitesimal generators of the scale change
$$\Lambda_s := \big(\frac{N}{2} - s\big) + x\cdot\grad.$$
For $s = 1$ we omit the subscript and write $\Lambda = \Lambda_1$.
We denote $\Lambda_\cE$, $\Lambda_\cF$ and $\Lambda_{\cE^*}$ the infinitesimal generators
of the scaling which is critical for a given norm, that is
$$\Lambda_\cE = (\Lambda, \Lambda_0),\quad \Lambda_\cF = (\Lambda_0, \Lambda_{-1}),\quad \Lambda_{\cE^*} = (\Lambda_{-1}, \Lambda_0).$$
We use the subscript $\cdot_\lambda$ to denote rescaling with characteristic length $\lambda$, critical for a norm which will be known from the context.

We introduce the following notation for some frequently used function spaces:
$X^s := \dot H_\tx{rad}^{s+1} \cap \dot H_\tx{rad}^1$ for $s \geq 0$,
$Y^k := H^k(1+|x|^k)$ for $k \in \bN$,
$\cE := \dot H^1_\tx{rad} \times L^2_\tx{rad}$, $\cF := L^2_\tx{rad} \times \dot H^{-1}_\tx{rad}$.
Notice that $\cE^* \simeq \dot H^{-1}_\tx{rad}\times L^2_\tx{rad}$ through the natural isomorphism given by the distributional pairing.
In the sequel we will omit the subscript and write $\dot H^1$ for $\dot H^1_\tx{rad}$ etc.
We denote $J := \begin{pmatrix}0 & \Id \\ -\Id & 0 \end{pmatrix}$; note that $J\cE^* = \cF$.

For a function space $\cA$, $O_\cA(m)$ denotes any $a \in \cA$ such that $\|a\|_\cA \leq Cm$ for some constant $C > 0$.
We denote $B_\cA(x_0, \eta)$ an open ball of center $x_0$ and radius $\eta$ in the space $\cA$.
If $\cA$ is not specified, it means that $\cA = \bR$.

For a radial function $g: \bR^N \to \bR$ and $r \geq 0$ we denote $g(r)$ the value of $g(x)$ for $|x| = r$.
\section{Sharp coercivity properties near $\bs W_\lambda$}
\label{sec:coer}
\subsection{Properties of the linearized operator}
\label{ssec:coer-lin}
Linearizing \eqref{eq:nlw} around $\bs W$, $\bs u = \bs W + \bs g$, one obtains
\begin{equation*}
  \partial_t \bs g = J\circ\vD^2 E(\bs W)\bs g = \begin{pmatrix} 0 & \Id \\ -L & 0\end{pmatrix} \bs g,
\end{equation*}
where $L$ is the Schr\"odinger operator
$$Lg := (-\Delta - f'(W))g.$$
Notice that $L(\Lambda W) = \dd\lambda\rstr_{\lambda = 1}\big(-\Delta W_\lambda - f(W_\lambda)\big) = 0$.
It is known that $L$ has exactly one strictly negative simple eigenvalue which we denote $-\nu^2$ (we take $\nu > 0$).
We denote the corresponding positive eigenfunction $\cY$, normalized so that $\|\cY\|_{L^2} = 1$.
By elliptic regularity $\cY$ is smooth and by Agmon estimates it decays exponentially.
Self-adjointness of $L$ implies that
\begin{equation}
  \label{eq:YLW}
  \la \cY, \Lambda W\ra = 0.
\end{equation}
  Fix $\cZ\in C_0^\infty$ such that
  \begin{equation}
    \label{eq:Z}
    \la \cZ, \Lambda W\ra > 0, \qquad \la \cZ, \cY\ra = 0.
  \end{equation}

  We have the following linear (localized) coercivity result, similar to \cite[Lemma 2.1]{MaMe15p}.
  \begin{lemma}
    \label{lem:coer-lin-coer}
    There exist constants $c, C > 0$ such that
    \begin{itemize}
      \item for all $g \in \dot H^1$ radially symmetric there holds
        \begin{equation}
          \label{eq:coer-lin-coer-1}
      \la g, Lg\ra = \int_{\bR^N}|\grad g|^2 \ud x - \int_{\bR^N}f'(W)|g|^2\ud x \geq c\int_{\bR^N}|\grad g|^2 \ud x -C\big(\la \cZ, g\ra^2 + \la \cY, g\ra^2\big),
    \end{equation}
  \item if $r_1 > 0$ is large enough, then for all $g \in \dot H^1_\tx{rad}$ there holds
    \begin{equation}
      \label{eq:coer-lin-coer-2}
      (1-2c)\int_{|x|\leq r_1}|\grad g|^2 \ud x + c\int_{|x|\geq r_1}|\grad g|^2\ud x - \int_{\bR^N}f'(W)|g|^2\ud x \geq -C\big(\la \cZ, g\ra^2 + \la \cY, g\ra^2\big),
    \end{equation}
  \item if $r_2 > 0$ is small enough, then for all $g \in \dot H^1_\tx{rad}$ there holds
    \begin{equation}
      \label{eq:coer-lin-coer-3}
      (1-2c)\int_{|x|\geq r_2}|\grad g|^2 \ud x + c\int_{|x|\leq r_2}|\grad g|^2\ud x - \int_{\bR^N}f'(W)|g|^2\ud x \geq -C\big(\la \cZ, g\ra^2 + \la \cY, g\ra^2\big).
    \end{equation}
\end{itemize}
  \end{lemma}
  \begin{proof}
    We will prove \eqref{eq:coer-lin-coer-2} and \eqref{eq:coer-lin-coer-3}.
    For a proof of \eqref{eq:coer-lin-coer-1} we refer to \cite[Lemma 6.1]{moi15p}, see also \cite[Proposition 5.5]{DM08} for a different formulation.

    We define the projections $\Pi_r, \Psi_r: \dot H^1 \to \dot H^1$:
    \begin{equation*}
      (\Pi_r g)(x) := \Big\{
        \begin{aligned}
          g(r) \qquad &\text{if }|x| \leq r, \\
          g(x) \qquad &\text{if }|x| \geq r,
        \end{aligned} \qquad
      (\Psi_r g)(x) := \Big\{
        \begin{aligned}
          g(x) - g(r) \qquad &\text{if }|x| \leq r, \\
          0 \qquad &\text{if }|x| \geq r
        \end{aligned}
    \end{equation*}
    (thus $\Pi_r + \Psi_r = \Id$).

    Applying \eqref{eq:coer-lin-coer-1} to $\Psi_{r_1}g$ with $c$ replaced by $3c$ and $C$ replaced by $\frac C2$ we get
    \begin{equation}
      \label{eq:lin-coer-dem-1}
      \begin{aligned}
        (1-2c)\int_{|x| \leq r_1}|\grad g|^2 \ud x &= (1-2c)\int_{\bR^N}|\grad(\Psi_{r_1}g)|^2 \ud x \\
        &\geq (1+c)\int_{\bR^N}f'(W)|\Psi_{r_1}g|^2\ud x - \frac C2\big(\la \cZ, \Psi_{r_1}g\ra^2 + \la \cY, \Psi_{r_1}g\ra^2\big).
    \end{aligned}
    \end{equation}
    By Sobolev and H\"older inequalities we have
    \begin{equation}
      \label{eq:lin-coer-dem-2}
      \begin{aligned}
      \int_{|x| \geq r_1}f'(W)|g|^2 \ud x &= \int_{|x| \geq r_1}f'(W)|\Pi_{r_1}g|^2 \ud x \\
      &\lesssim \|f'(W)\|_{L^\frac N2(|x| \geq r_1)}\cdot \|\Pi_{r_1}g\|_{\dot H^1}^2 \leq \frac c4 \int_{|x| \geq r_1}|\grad g|^2 \ud x
    \end{aligned}
    \end{equation}
    if $r_1$ is large enough.

    In the region $|x| \leq r_1$ we apply the pointwise inequality
    \begin{equation}
      \label{eq:lin-coer-dem-3}
      |g(x)|^2 \leq (1+c)|(\Psi_{r_1}g)(x)|^2 + (1 + c^{-1})|g(r_1)|^2,\qquad |x| \leq r_1.
    \end{equation}
    Recall that by the Strauss Lemma \cite{strauss77}, for a radial function $g$ there holds
    \begin{equation*}
      |g(r_1)| \lesssim \|\Pi_{r_1}g\|_{\dot H^1}\cdot r_1^{-\frac{N-2}{2}}.
    \end{equation*}
    Since $f'(W(r)) \sim r^{-4}$ as $r \to +\infty$, we have
    \begin{equation*}
      \int_{|x|\leq r_1}f'(W)\ud x \ll r_1^{N-2},\qquad \text{as }r_1 \to +\infty,
    \end{equation*}
    hence
    \begin{equation}
      \label{eq:lin-coer-dem-4}
      \int_{|x|\leq r_1}f'(W)\cdot(1 + c^{-1})|g(r_1)|^2\ud x \leq \frac c4 \int_{|x| \geq r_1}|\grad g|^2 \ud x
    \end{equation}
    if $r_1$ is large enough.

    Estimates \eqref{eq:lin-coer-dem-2}, \eqref{eq:lin-coer-dem-3} and \eqref{eq:lin-coer-dem-4} yield
    \begin{equation}
      \label{eq:lin-coer-dem-5}
      \int_{\bR^N} f'(W)|g|^2 \ud x \leq (1+c)\int_{\bR^N} f'(W)|(\Psi_{r_1}g)(x)|^2\ud x + \frac{c}{2} \int_{|x| \geq r_1}|\grad g|^2 \ud x.
    \end{equation}
    Using the fact that $\cY \in L^1 \cap L^\frac{2N}{N+2}$ we obtain
    \begin{equation*}
      |\la \cY, \Pi_{r_1}g\ra| \lesssim \|\Pi_{r_1}g\|_{\dot H^1}\cdot r_1^{-\frac{N-2}{2}} + \int_{|x| \geq r_1}\cY|g|\ud x \lesssim (r_1^{-\frac{N-2}{2}} + \|\cY\|_{L^\frac{2N}{N+2}(|x| \geq r_1)})\|\Pi_{r_1} g\|_{\dot H^1},
    \end{equation*}
    hence
    \begin{equation}
      \label{eq:lin-coer-dem-6}
      \frac C2\la \cY, \Psi_{r_1}g\ra^2 \leq C\la \cY, g\ra^2 + C\la \cY, \Pi_{r_1} g\ra^2 \leq C\la \cY, g\ra^2 + \frac c4\int_{|x| \geq r_1}|\grad g|^2 \ud x,
    \end{equation}
    provided that $r_1$ is chosen large enough. Similarly,
    \begin{equation}
      \label{eq:lin-coer-dem-7}
      \frac C2\la \cZ, \Psi_{r_1}g\ra^2 \leq C\la \cZ, g\ra^2 + C\la \cZ, \Pi_{r_1} g\ra^2 \leq C\la \cZ, g\ra^2 + \frac c4\int_{|x| \geq r_1}|\grad g|^2 \ud x.
    \end{equation}
    Estimate \eqref{eq:coer-lin-coer-2} follows from \eqref{eq:lin-coer-dem-1}, \eqref{eq:lin-coer-dem-5}, \eqref{eq:lin-coer-dem-6} and \eqref{eq:lin-coer-dem-7}.

    We turn to the proof of \eqref{eq:coer-lin-coer-3}.
    Applying \eqref{eq:coer-lin-coer-1} to $\Pi_{r_2}g$ with $c$ replaced by $3c$ and $C$ replaced by $\frac C2$ we get
    \begin{equation}
      \label{eq:lin-coer-dem-11}
      \begin{aligned}
        (1-3c)\int_{|x| \geq r_2}|\grad g|^2 \ud x &= (1-3c)\int_{\bR^N}|\grad(\Pi_{r_2}g)|^2 \ud x \\
        &\geq \int_{\bR^N}f'(W)|\Pi_{r_2}g|^2\ud x - \frac C2\big(\la \cZ, \Pi_{r_2}g\ra^2 + \la \cY, \Pi_{r_2}g\ra^2\big).
    \end{aligned}
    \end{equation}
    By Sobolev and H\"older inequalities we have for $r_2$ small enough
    \begin{equation}
      \label{eq:lin-coer-dem-12}
      \int_{|x| \leq r_2}f'(W)|g|^2 \ud x \leq \frac c2 \int_{\bR^N}|\grad g|^2 \ud x.
    \end{equation}
    By definition of $\Pi_r$ there holds
    \begin{equation*}
      \int_{|x| \geq r_2}f'(W)|g|^2 \ud x \leq \int_{\bR^N}f'(W)|\Pi_{r_2}g|^2\ud x,
    \end{equation*}
    hence \eqref{eq:lin-coer-dem-11} and \eqref{eq:lin-coer-dem-12} imply
    \begin{equation}
      \label{eq:lin-coer-dem-13}
      (1-2c)\int_{|x| \geq r_2}|\grad g|^2 \ud x + \frac c2\int_{|x| \leq r_2}|\grad g|^2 \ud x \geq
      \int_{\bR^N}f'(W)|g|^2\ud x - \frac C2\big(\la \cZ, \Pi_{r_2}g\ra^2 + \la \cY, \Pi_{r_2}g\ra^2\big).
    \end{equation}
    Using the fact that $\cY \in L^\frac{2N}{N+2}$ we obtain
    \begin{equation*}
      |\la \cY, \Psi_{r_2}g\ra| \lesssim \int_{|x| \leq r_2}\cY|g|\ud x \lesssim \|\cY\|_{L^\frac{2N}{N+2}(|x| \leq r_2)}\|\Psi_{r_2} g\|_{\dot H^1},
    \end{equation*}
    hence
    \begin{equation}
      \label{eq:lin-coer-dem-14}
      \frac C2\la \cY, \Pi_{r_2}g\ra^2 \leq C\la \cY, g\ra^2 + C\la \cY, \Psi_{r_2} g\ra^2 \leq C\la \cY, g\ra^2 + \frac c4\int_{|x| \leq r_2}|\grad g|^2 \ud x,
    \end{equation}
    provided that $r_2$ is chosen small enough. Similarly,
    \begin{equation}
      \label{eq:lin-coer-dem-15}
      \frac C2\la \cZ, \Pi_{r_2}g\ra^2 \leq C\la \cZ, g\ra^2 + C\la \cZ, \Psi_{r_2} g\ra^2 \leq C\la \cZ, g\ra^2 + \frac c4\int_{|x| \leq r_2}|\grad g|^2 \ud x.
    \end{equation}
    Estimate \eqref{eq:coer-lin-coer-3} follows from \eqref{eq:lin-coer-dem-13}, \eqref{eq:lin-coer-dem-14} and \eqref{eq:lin-coer-dem-15}.

  \end{proof}
  
We define
\begin{equation}
  \label{eq:Y}
  \ym := \big(\frac 1\nu\cY, -\cY\big),\qquad \yp := \big(\frac 1\nu\cY, \cY\big),
\end{equation}
\begin{equation}
  \label{eq:a}
  \alpha^- := \frac{\nu}{2}J\cY^+ = \frac 12(\nu\cY, -\cY),\qquad \alpha^+ := -\frac{\nu}{2}J\cY^- =\frac 12(\nu\cY, \cY).
\end{equation}

We have $J\circ\vD^2 E(\bs W) = \begin{pmatrix} 0 & \Id \\ -L & 0\end{pmatrix}$. A short computation shows that
  \begin{equation}
    \label{eq:eigenvect}
    J\circ\vD^2 E(\bs W)\ym = -\nu \ym,\qquad J\circ\vD^2 E(\bs W)\yp = \nu \yp
  \end{equation}
  and
  \begin{equation}
    \label{eq:eigencovect}
    \la\alpha^-, J\circ\vD^2 E(\bs W)\bs g\ra = -\nu\la\alpha^-, \bs g\ra,\qquad \la\alpha^+, J\circ\vD^2 E(\bs W)\bs g\ra = \nu\la\alpha^+, \bs g\ra,\qquad \forall \bs g\in\cE.
  \end{equation}
  We will think of $\alpha^-$ and $\alpha^+$ as linear forms on $\cE$.
  Notice that $\la \alpha^-, \ym\ra = \la \alpha^+, \yp\ra = 1$ and $\la \alpha^-, \yp\ra = \la \alpha^+, \ym\ra = 0$.

  The rescaled versions of these objects are
  \begin{equation}
    \label{eq:Yl}
    \yml := \big(\frac 1\nu\cY_\lambda, -\cY_\uln\lambda\big),\qquad \ypl := \big(\frac 1\nu\cY_\lambda, \cY_\uln\lambda\big),
  \end{equation}
  \begin{equation}
    \label{eq:al}
    \alpha^-_\lambda := \frac{\nu}{2\lambda}J\cY_\lambda^+ = \frac 12\big(\frac{\nu}{\lambda}\cY_\uln\lambda, -\cY_\uln\lambda\big),\qquad \alpha^+_\lambda := -\frac{\nu}{2\lambda}J\cY_\lambda^- = \frac 12\big(\frac{\nu}{\lambda}\cY_\uln\lambda, \cY_\uln\lambda\big).
  \end{equation}
  The scaling is chosen so that $\la\alpha_\lambda^-, \ym_\lambda\ra = \la\alpha_\lambda^+, \yp_\lambda\ra = 1$.
  We have
  \begin{equation}
    \label{eq:eigenvectl}
    J\circ\vD^2 E(\bs W_\lambda)\ym_\lambda = -\frac{\nu}{\lambda} \ym_\lambda,\qquad J\circ\vD^2 E(\bs W_\lambda)\yp_\lambda = \frac{\nu}{\lambda} \yp_\lambda
  \end{equation}
  and
  \begin{equation}
    \label{eq:eigencovectl}
    \la\alpha_\lambda^-, J\circ\vD^2 E(\bs W_\lambda)\bs g\ra = -\frac{\nu}{\lambda}\la\alpha_\lambda^-, \bs g\ra,
    \qquad \la\alpha_\lambda^+, J\circ\vD^2 E(\bs W_\lambda)\bs g\ra = \frac{\nu}{\lambda}\la\alpha_\lambda^+, \bs g\ra,\qquad \forall \bs g\in\cE.
  \end{equation}

  As a standard consequence of \eqref{eq:coer-lin-coer-1}, we obtain the following:
\begin{lemma}
    \label{lem:coer}
    There exists a constant $\eta > 0$ such that if $\|\bs V - \bs W_\lambda\|_\cE < \eta$, then for all $\bs g \in \cE$ such that $\la \cZ_\uln\lambda, g\ra = 0$ there holds
    \begin{equation*}
      \frac 12 \la \vD^2 E(\bs V)\bs g, \bs g\ra +2\big(\la \alpha^-_\lambda, \bs g\ra^2 + \la \alpha^+_\lambda, \bs g\ra^2\big) \gtrsim \|\bs g\|_\cE^2.
    \end{equation*}
  \end{lemma}
  \begin{proof}
    For $N \in \{3, 4, 5\}$ see \cite[Lemma 2.2]{moi15p-2}. For $N \geq 6$ the same proof is valid, once we notice that $\|f'(V) - f'(W_\lambda)\|_{L^\frac N2} \leq f'(\|V - W_\lambda\|_{\dot H^1})$.
  \end{proof}
 
We now turn to the proofs of various energy estimates for the linear group generated by
  $$
  A := J\circ\vD^2 E(\bs W) = \begin{pmatrix} 0 & \Id \\ -L & 0\end{pmatrix}.
    $$
on its invariant subspaces,
which will be needed in Subsection~\ref{ssec:coer-constr}. This is much in the spirit of \cite[Section 2]{BoWa97}.

It follows from \eqref{eq:eigencovect} that the \emph{centre-stable subspace} $\cX_\tx{cs} := \ker \alpha^+$,
the \emph{centre-unstable subspace} $\cX_\tx{cu} := \ker \alpha^-$ and the \emph{centre subspace}
$\cX_\tx{c} := \cX_\tx{cs} \cap \cX_\tx{cu}$ are invariant subspaces of the operator $A$.
Notice that $\la\alpha^-, \ym\ra = \la\alpha^+, \yp\ra = 1$, $\cE = \cX_\tx{cs} \oplus \{a\yp\} = \cX_\tx{cu} \oplus \{a\ym\}$,
$\cX_\tx{cs} = \cX_\tx{c} \oplus \{a\ym\}$, $\cX_\tx{cu} = \cX_\tx{c} \oplus \{a\yp\}$.

We define $\cX_\tx{cc} := \{\bs v = (v, \dot v) \in \cX_\tx{c}\ |\ \la \cZ, v\ra = 0\}$.
\begin{lemma}
  \label{lem:normA}
  Let $k \in \bN$. There exists constants $1 = a_0 > a_1 > \ldots > a_k > 0$ such that the norm $\|\cdot\|_{A, k}$ defined by the following formula:
  \begin{equation}
    \label{eq:normA}
    \|\bs v\|_{A, k}^2 := \sum_{j=0}^k a_j \big(\la v, L^{j+1}v\ra + \la \dot v, L^j\dot v\ra\big)
  \end{equation}
  satisfies $\|\bs v\|_{X^k \times H^k} \sim \|\bs v\|_{A, k}$ for all $\bs v = (v, \dot v) \in (X^k\times H^k) \cap \cX_\tx{cc}$.
\end{lemma}
\begin{proof}
  We proceed by induction. For $k = 0$ we have
  $$
  \|\bs v\|_{A, 0} = \sqrt{\la v, L v\ra + \la \dot v, \dot v\ra} = \sqrt{\la \vD^2 E(\bs W)\bs v, \bs v\ra}.
  $$
  By Lemma~\ref{lem:coer}, this norm is equivalent to $\|\cdot\|_\cE$ on $\cE \cap \cX_\tx{cc}$.

  To check the induction step, one should show that for any $k > 0$ there exists $a_1, a_2 > 0$ such that
  \begin{equation}
    \label{eq:equiv-n-1}
    \| v\|_1^2 := \|v\|_{X^{k-1}}^2 + a_1\la v, L^{k+1}v\ra \gtrsim \|v\|_{X^k}^2
  \end{equation}
  and
  \begin{equation}
    \label{eq:equiv-n-2}
    \| \dot v\|_2^2 := \|\dot v\|_{H^{k-1}}^2 + a_2\la \dot v, L^k\dot v\ra \gtrsim \|\dot v\|_{H^k}^2.
  \end{equation}
  To prove \eqref{eq:equiv-n-2} notice that
  \begin{equation*}
    L^k = (-\Delta)^k + (\text{terms with at most $2k-2$ derivatives}).
  \end{equation*}
  Integrating by parts all the terms except for the first one we arrive at expressions of the form $\int V\cdot\partial^i v\cdot\partial^j v\ud x$
  where $V$ is bounded and $i, j \leq k-1$. All these expressions are controlled by $\|\cdot \|_{H^{k-1}}^2$.

  The proof of \eqref{eq:equiv-n-1} is almost the same. The only problem are the terms of the form $\int V\cdot \grad v\cdot v\ud x$
  and $\int V\cdot |v|^2\ud x$. As the potential decreases at least as $f'(W)$,
  by Hardy inequality these terms are controlled by $\|v\|_{\dot H^1}$.
\end{proof}
We will denote $\la \cdot, \cdot \ra_{A,k}$ the scalar product associated with the norm $\|\cdot\|_{A,k}$.

We define the projections:
\begin{equation*}
  \pi_s \bs v := \la\alpha^-, \bs v\ra\cY^-,\qquad \pi_{cu} := \Id - \pi_s.
\end{equation*}
We denote $\pi_{cc}$ the projection of $\cX_\tx{c}$ on $\cX_\tx{cc}$ in the direction $\Lambda_\cE \bs W$.
These projections are continuous linear operators on $\cE$ as well as on $X^k \times H^k$ for $k > 0$.
\begin{proposition}
  \label{prop:group}
  The operator $A$ generates a strongly continuous group on $X^k \times H^k$ denoted $\eee^{tA}$. Moreover, the following bounds are true:
  \begin{align}
    \bs v_0 \in (X^k\times H^k)\cap\cX_\tx{s} \quad\Rightarrow\quad &\|\eee^{tA}\bs v_0\|_{X^k \times H^k} \lesssim \eee^{-\nu t}\|\bs v_0\|_{X^k \times H^k} &\text{ for }t \geq 0, \label{eq:group-s} \\
    \bs v_0 \in (X^k\times H^k)\cap\cX_\tx{cu} \quad\Rightarrow\quad &\|\eee^{-tA}\bs v_0\|_{X^k \times H^k} \lesssim (1+t)\|\bs v_0\|_{X^k\times H^k} &\text{ for }t \geq 0,\label{eq:group-cu} \\
    \bs v_0 \in X^k\times H^k \quad\Rightarrow\quad &\|\eee^{-tA}\bs v_0\|_{X^k \times H^k} \lesssim \eee^{\nu t}\|\bs v_0\|_{X^k\times H^k} &\text{ for }t \geq 0.\label{eq:group-all}
  \end{align}
\end{proposition}
\begin{proof}
  It suffices to analyse the restriction to the invariant subspace $\cX_\tx{c}$.
  Take $\bs v_0 \in \cX_\tx{c}$ and decompose $\bs v_0 = l_0 \Lambda_\cE \bs W + \bs w_0$, $\bs w_0 \in \cX_\tx{cc}$
  (notice that $\Lambda_\cE \bs W \in X^k \times H^k$).
  It can be checked that the operator $B := \pi_{cc}\circ A$ is skew-adjoint for the scalar product $\la\cdot, \cdot\ra_{A, k}$,
  hence it generates a unitary group $\bs w(t) = \eee^{tB}\bs w_0$ by the Stone theorem.
  Let $l(t)$ be defined by the formula
  \begin{equation}
    \label{eq:rown-l}
    l(t) = l_0 + \int_0^t\frac{\la \cZ, \dot w(t)\ra}{\la \cZ, \Lambda W\ra}\ud t.
  \end{equation}
  Set $\bs v(t) = \bs w(t) + l(t)\Lambda_\cE \bs W$. This defines a linear group and
  \begin{equation*}
    \lim_{t\to 0}\frac{1}{t}(\bs v(t) - \bs v_0) = B\bs w_0 + l'(0)\Lambda_\cE \bs W = B\bs v_0 + \frac{\la \cZ, \dot w_0\ra}{\la \cZ, \Lambda W\ra}\Lambda_\cE \bs W = A\bs v_0,
  \end{equation*}
  hence $\bs v(t) = \eee^{tA}\bs v_0$.

  Estimate \eqref{eq:group-s} follows immediately from the fact that $\cY^-$ is an eigenfunction of $A$ with eigenvalue $-\nu$.
  Analogously, in \eqref{eq:group-cu} we can assume that $\bs v_0 \in \cX_\tx{c}$ (the unstable mode decreases exponentially for negative times).
  By the equivalence of norms and the fact that the group generated by $B$ is unitary for the norm $\|\cdot\|_{A,k}$,
  \begin{equation}
    \label{eq:bound-w}
    \|\bs w(t)\|_{X^k\times H^k} \lesssim \|\bs v_0\|_{X^k\times H^k}\quad\text{for all }t,
  \end{equation}
    hence it suffices to bound $l(t)$.
    Using \eqref{eq:bound-w} and the fact that $|l_0| \lesssim \|\bs v_0\|_{X^k \times H^k}$ we get from \eqref{eq:rown-l} that $|l(t)| \lesssim (1+|t|)\|\bs v_0\|_{X^k\times H^k}$.

    Estimate \eqref{eq:group-all} follows easily from \eqref{eq:group-cu}.
\end{proof}

\begin{remark}
  \label{rem:group}
  The factor $|t|$ in \eqref{eq:group-cu} is necessary, for example in dimension $N = 5$ we have a solution $\bs v(t) = (t\Lambda W, \Lambda W)$.
\end{remark}

It is possible to finish the construction for example in the space $X^1 \times H^1$. However, later we will need some information on the spatial decay of the constructed functions,
which forces us to use weighted spaces. We define
\begin{equation*}
  \|v\|_{Y^k} := \|(1+|x|^k)v\|_{H^k}.
\end{equation*}
One may check by induction on $j = 0, 1, \ldots, k$ that
\begin{equation*}
  \|(1+|x|^k)v\|_{H^j}^2 \sim \int (1+|x|^{2k})(|v|^2+|\grad^j v|^2)\ud x,
\end{equation*}
in particular
\begin{equation*}
  \|v\|_{Y^k}^2 \sim \int (1+|x|^{2k})(|v|^2+|\grad^k v|^2)\ud x.
\end{equation*}

\begin{lemma}
  \label{lem:group}
  Let $k\in \bZ, k\geq 0$. The following bounds are true for $t \geq 0$:
  \begin{align}
    \bs v_0 \in (Y^{k+1}\times Y^k)\cap\cX_\tx{s} \quad\Rightarrow\quad &\|\eee^{tA}\bs v_0\|_{Y^{k+1}\times Y^k} \lesssim \eee^{-\nu t}\|\bs v_0\|_{Y^{k+1}\times Y^k}, \label{eq:group-s-Y} \\
    \bs v_0 \in (Y^{k+1}\times Y^k)\cap\cX_\tx{cu} \quad\Rightarrow\quad &\|\eee^{-tA}\bs v_0\|_{Y^{k+1}\times Y^k}\lesssim (1+t^{\frac{k(k+1)}{2} + 1})\|\bs v_0\|_{Y^{k+1}\times Y^k},\label{eq:group-cu-Y} \\
    \bs v_0 \in Y^{k+1}\times Y^k \quad\Rightarrow\quad &\|\eee^{-tA}\bs v_0\|_{Y^{k+1}\times Y^k} \lesssim \eee^{\nu t}\|\bs v_0\|_{Y^{k+1}\times Y^k}.\label{eq:group-all-Y}
  \end{align}
\end{lemma}
\begin{proof}
  The proof of \eqref{eq:group-s-Y} and \eqref{eq:group-all-Y} is the same as in Proposition~\ref{prop:group}, once we recall that $\cY^- \in Y^{k+1}\times Y^k$.
  In order to prove \eqref{eq:group-cu-Y}, write $\bs v(t) = \eee^{-tA}\bs v_0$, so that $\partial_t \bs v = -A\bs v = (-\dot v, -\Delta v - f'(W)v)$, hence
  \begin{equation*}
    \begin{aligned}
      \frac 12 \dd t\int(1+|x|^{2k})(|\dot v|^2+|\grad v|^2)\ud x &= -\int(1+|x|^{2k})\big((\Delta v + f'(W)v)\cdot\dot v+ \grad\dot v\cdot\grad v\big) \ud x \\
      &=\int\grad(|x|^{2k})\cdot\grad v\cdot\dot v + (1+|x|^{2k})(f'(W)v)\cdot\dot v\ud x
    \end{aligned}
  \end{equation*}
  (we have integrated by parts between the first and the second line).
  Notice that $xf'(W) \in L^N$, hence by H\"older and Sobolev $\|xf'(W)v\|_{L^2} \lesssim \|v\|_{L^\frac{2N}{N-2}} \lesssim \|\grad v\|_{L^2}$, thus
  \begin{equation}
    \label{eq:energie-poids-0}
    \Big|\dd t\int(1+|x|^{2k})(|\dot v|^2+|\grad v|^2)\ud x \Big| \lesssim \int(1+|x|^{2k-1})(|\dot v|^2+|\grad v|^2)\ud x.
  \end{equation}
  
  Analogously, from
  \begin{equation*}
    \begin{aligned}
      &\frac 12 \dd t\int(1+|x|^{2k})(|\grad^k \dot v|^2+|\grad^{k+1} v|^2)\ud x \\= &-\int(1+|x|^{2k})\big(\grad^k(\Delta v + f'(W)v)\cdot\grad^{k}\dot v+ \grad^{k+1}\dot v\cdot\grad^{k+1} v\big) \ud x \\
      = &\int\grad(|x|^{2k})\cdot\grad^{k+1}v\cdot\grad^k\dot v + (1+|x|^{2k})\grad^{k}(f'(W)v)\cdot\grad^{k}\dot v\ud x
    \end{aligned}
  \end{equation*}
  we deduce
  \begin{equation}
    \label{eq:energie-poids-k}
    \Big|\dd t\int(1+|x|^{2k})(|\grad^k \dot v|^2+|\grad^{k+1} v|^2)\ud x \Big| \lesssim \int(1+|x|^{2k-1})(|\dot v|^2 + |\grad^k \dot v|^2+|\grad v|^2 + |\grad^{k+1} v|^2)\ud x.
  \end{equation}
  
  Using \eqref{eq:energie-poids-0}, \eqref{eq:energie-poids-k} and H\"older we obtain
  \begin{equation*}
    \begin{aligned}
    &\Big| \dd t \int(1+|x|^{2k})(|\dot v|^2+|\grad^k \dot v|^2 + |\grad v|^2+|\grad^{k+1}v|^2)\ud x\Big| \\
    \lesssim &\Big(\int(1+|x|^{2k})(|\dot v|^2+|\grad^k \dot v|^2 + |\grad v|^2+|\grad^{k+1}v|^2)\ud x\Big)^\frac{2k-1}{2k}\cdot\|\bs v\|_{X^k\times H^k}^\frac 1k,
  \end{aligned}
  \end{equation*}
  which gives, using \eqref{eq:group-cu} and integrating,
  \begin{equation*}
    \int(1+|x|^{2k})(|\dot v|^2+|\grad^k \dot v|^2 + |\grad v|^2+|\grad^{k+1}v|^2)\ud x \lesssim (1+t^{k(k+1)})\|\bs v_0\|_{Y^{k+1}\times Y^k}^2.
  \end{equation*}
  Now we can easily bound the $L^2$ term by Schwarz inequality:
  \begin{equation*}
    \Big| \dd t \int(1+|x|^{2k})|v|^2\ud x\Big| \lesssim \Big(\int(1+|x|^{2k})|v|^2\ud x\Big)^\frac 12\cdot \Big(\int(1+|x|^{2k})|\dot v|^2\ud x\Big)^\frac 12,
  \end{equation*}
  which leads to
  \begin{equation*}
    \int(1+|x|^{2k})|v|^2\ud x \lesssim (1+t^{k(k+1)+2})\|\bs v_0\|_{Y^{k+1}\times Y^k}^2.
  \end{equation*}
\end{proof}

We fix $k \in \bN$ large enough. For $\wt \nu > 0$ the space $BC_{\wt\nu}$ is defined as the space of continuous functions $\bs v: [0, +\infty) \to Y^{k+1}\times Y^k$ with the norm
  \begin{equation*}
    \|\bs v\|_{BC_{\wt\nu}} := \sup_{t\in[0, +\infty)}\eee^{\wt\nu t}\|\bs v(t)\|_{Y^{k+1}\times Y^k}.
  \end{equation*}
  \begin{lemma}
    \label{lem:nonhom}
    If $\wt \nu \in (0, \nu)$, then for any $\bs w \in BC_{\wt\nu}$ the equation
    \begin{equation}
      \label{eq:nonhom-eq}
      \partial_t \bs v(t) = A\bs v(t) + \bs w(t)
    \end{equation}
    has a unique solution $\bs v  = K\bs w\in BC_{\wt\nu}$ such that $\la\alpha^-, \bs v(0)\ra = 0$.
    
    
    In addition, $K$ is a bounded linear operator on $BC_{\wt\nu}$.
  \end{lemma}
  \begin{proof}
%
    Suppose that $\bs v \in BC_{\wt\nu}$ verifies \eqref{eq:nonhom-eq}. Denote $\bs v_0 = \bs v(0)$.
    From the Duhamel formula we obtain
    \begin{equation}
      \label{eq:duhamel-2}
      \bs v(t) = \eee^{tA}\bs v_0 + \int_0^t\eee^{(t-\tau)A}\bs w(\tau)\ud \tau\ \Rightarrow\ \eee^{-tA}\pi_{cu} \bs v(t) = \pi_{cu}\bs v_0 + \int_0^t \eee^{-\tau A}\pi_{cu} \bs w(\tau)\ud \tau.
    \end{equation}
    By assumption, $\|\bs v(t)\|_{Y^{k+1}\times Y^k} \lesssim \eee^{-\wt\nu}$, hence from \eqref{eq:group-cu} we infer $\eee^{-tA}\pi_{cu}\bs v(t) \lesssim (1+t^\kappa)\eee^{-\wt\nu}$, $\kappa := \frac 12 k(k+1) + 1$.
    Passing to the limit $t \to +\infty$ yields
    \begin{equation*}
      \pi_{cu} \bs v_0 = -\int_0^{+\infty}\eee^{-\tau A}\pi_{cu} \bs w(\tau)\ud \tau.
    \end{equation*}
    If we require $\la\alpha^-, \bs v_0\ra = 0$, this determines uniquely $\bs v_0 = \pi_{cu} \bs v_0$, hence, using \eqref{eq:duhamel-2},
    \begin{equation*}
      \bs v(t) = K\bs w(t) = -\int_t^{+\infty}\eee^{(t-\tau) A}\pi_{cu} \bs w(\tau)\ud \tau + \int_0^t\eee^{(t-\tau) A}\pi_{s} \bs w(\tau)\ud \tau.
    \end{equation*}
    From \eqref{eq:group-s} and \eqref{eq:group-cu} we obtain
    \begin{equation*}
      \begin{aligned}
      \|K \bs w(t)\|_{Y^{k+1}\times Y^k} &\lesssim \|\bs w\|_{BC_{\wt\nu}}\cdot\Big(\int_t^{+\infty}(1+(\tau-t)^\kappa)\eee^{-\wt \nu\tau}\ud \tau + \int_0^t \eee^{(\tau-t)\nu}\eee^{-\wt \nu\tau}\ud \tau\Big) \\
      &\lesssim \|\bs w\|_{BC_{\wt\nu}}\cdot \eee^{-\wt \nu t},
    \end{aligned}
    \end{equation*}
    so $K$ is a bounded operator.
  \end{proof}
  \begin{remark}
    \label{rem:stable-all}
    By linearity the unique solution of \eqref{eq:nonhom-eq} such that $\la\alpha^-, \bs v(0))\ra= a$ is $\bs v(t) = (K\bs w)(t) + \eee^{-\nu t}a\cY^-$.
  \end{remark}
\subsection{Construction of $\bs U_\lambda^a$}
\label{ssec:coer-constr}
As noted earlier, the functions $\bs U_\lambda^a$ were constructed in \cite[Section~6]{DM08}.
However, the construction given there does not give the additional regularity or decay,
which is required in the present paper. For this reason, we provide here a different construction.
Our construction is an adaptation of a classical ODE proof, see for instance \cite[Chapter 3.6]{chow-hale}.

We denote $$\cR(v) := f(W+v) - f(W) - f'(W)v.$$
\begin{lemma}
  \label{lem:stable-mani}
  Let $\wt \nu \in (0, \nu)$. There exist $\eta > 0$ such that for every $b\in\bR, |b| < \eta$
  there is a unique solution $\bs v = \bs v^b \in BC_{\wt \nu}$ of the equation
  \begin{equation}
    \label{eq:stable-eq}
    \partial_t \bs v(t) = A\bs v(t) + \cR(\bs v(t))
  \end{equation}
  such that $\la\alpha^-, \bs v(0)\ra = b$ and $\|\bs v\|_{BC_{\wt\nu}} < \eta$. Moreover, $\bs v^b$ is analytic with respect to $b$.
\end{lemma}
\begin{proof}
  Let $T: BC_{\wt\nu}\times \bR \to BC_{\wt\nu}$ be defined by the formula
  \begin{equation*}
    T(\bs v, b) := \eee^{-\nu t}b\cY^- + K(\cR(\bs v)),
  \end{equation*}
  where $K$ is the operator from Lemma~\ref{lem:nonhom}.
  Then $\bs v$ is a solution of \eqref{eq:stable-eq} if and only if $T(\bs v, b) = \bs v$ (see Remark \ref{rem:stable-all}).
  
  It follows from Lemma~\ref{lem:nonlin-anal} that on some neighbourhood of the origin $T$ is analytic and a uniform contraction with respect to $\bs v$,
  hence the conclusion follows from the Uniform Contraction Principle, cf. \cite[Theorem 2.2]{chow-hale}.
\end{proof}
\begin{proposition}
  \label{prop:Ua}
  For any $k \in \bN$ there exists $\eta > 0$ and an analytic function
  $$
  (-\eta, \eta) \owns a \mapsto \bs U^a -\bs W \in Y^{k+1}\times Y^k
$$
such that
\begin{align}
  \bs U^0 &= \bs W, \label{eq:Ua-0} \\
  \partial_a \bs U^a\vert_{a=0} &= \cY^-, \label{eq:daUa-0} \\
  -\nu a \partial_a \bs U^a &= J\circ \vD E(\bs U^a).\label{eq:Ua-diff-eq}
\end{align}
\end{proposition}

\begin{proof}
Evaluation at $t = 0$ is a bounded linear operator from $BC_{\wt \nu}$ to $Y^{k+1}\times Y^k$, hence
$\{\bs v^b(0): b \in (-\eta, \eta)\}$ defines an analytic curve in $Y^{k+1}\times Y^k$.
We have $\|\bs v^b\|_{BC_{\wt\nu}} \lesssim |b|$, so $\|\cR(\bs v^b)\|_{BC_{\wt\nu}} \lesssim |b|^2$.
By construction, $\bs v^b$ satisfies the equation
\begin{equation*}
  \bs v^b = b\eee^{-\nu t}\cY^- + K(\cR(\bs v^b)),
\end{equation*}
hence $\|\bs v^b - b\eee^{-\nu t}\cY^-\|_{BC_{\wt\nu}} \lesssim |b|^2$, in particular
\begin{equation*}
  \|\bs v^b(0) - b \cY^-\|_{Y^{k+1}\times Y^k} \lesssim |b|^2.
\end{equation*}
Because of uniqueness in Lemma~\ref{lem:stable-mani}, the set $\{\bs v^b(0): b \in (-\eta, \eta)\}$ is forward invariant if $\eta$ is small enough,
hence for all $b\in (-\eta, \eta)$ there exists a function $b(t)$ such that $\bs v^b(t) = \bs v^{b(t)}(0)$. The value of $b(t)$ is determined
by the condition $$\la\alpha^-, \bs v^b(t)\ra = b(t).$$
Differentiating in time this condition we find
\begin{equation*}
  b'(t) = \dd t\la\alpha^-, \bs v^b(t)\ra = \big\la\alpha^-, J\circ \vD E(\bs W + \bs v^b(t))\big\ra = \big\la\alpha^-, J\circ \vD E(\bs W + \bs v^{b(t)}(0))\big\ra = \psi(b(t)),
\end{equation*}
where $\psi$ is an analytic function, $\psi(0) = 0$ and $\psi'(0) = -\nu$. By Lemma~\ref{lem:hartman-1}, there exists an analytic change of variable $a = a(b)$
which transforms the equation $b'(t) = \psi(b(t))$ into $a'(t) = -\nu a(t)$ and such that $a(0) = 0$, $a'(0) = 1$. We define
\begin{equation*}
  \bs U^a := \bs W + \bs v^{b(a)}(0).
\end{equation*}

\end{proof}

We will denote $\bs U^a_\lambda := (\bs U^a)_\lambda$. Rescaling \eqref{eq:Ua-0}, \eqref{eq:daUa-0} and \eqref{eq:Ua-diff-eq} we obtain
\begin{align}
  \bs U^0_\lambda &= \bs W_\lambda, \\
  \partial_a \bs U^a_\lambda\vert_{a = 0} &= \cY^-_\lambda, \\
  \partial_a \bs U^a_\lambda &= -\frac{\lambda}{\nu a}J\circ\vD E(\bs U^a_\lambda). \label{eq:coer-daU}
\end{align}
\begin{remark}
  \label{rem:Wpm}
Note that \eqref{eq:coer-daU} implies that $\bs u(t) = \bs U_\lambda^{\pm \exp(-\frac{\nu}{\lambda} t)}$ is a solution of \eqref{eq:nlw} for large $t$. These are precisely the solutions $\bs W_\lambda^\pm$ mentioned in the Introduction.
\end{remark}
\subsection{Modulation near the stable manifold}
\label{ssec:coer-mod}
The results of this subsection will not be directly used in the proof of Theorem~\ref{thm:deux-bulles}.
We include them in the paper for their own interest and because the proofs introduce in a simple setting the main technical ideas
required in Section~\ref{sec:bulles}.

It is well known since the work of Payne and Sattinger \cite{PaSa75} that solutions of energy $< E(\bs W)$ leave a neighbourhood
of the family of stationary states. The aim of this subsection is to describe
an explicit local mechanism leading to this phenomenon, which is robust enough not to be significantly altered by the presence
of the second bubble (as will be the case in Section~\ref{sec:bulles}).

Note that nothing specific to the wave equation has been used so far, hence one might expect that all the proofs
of Section~\ref{sec:coer} should apply to similar (not necessarily critical) models in the presence of one instability direction
near a stationary state.
\begin{lemma}
  \label{lem:coer-mod}
  Let $\delta_0 > 0$ be sufficiently small. For any $0 \leq \delta \leq \delta_0$ there exists $0 \leq \eta = \eta(\delta) \sto{\delta \to 0} 0$ such that
  if $\bs u: (t_1, t_2) \to \cE$ is a solution of \eqref{eq:nlw} satisfying for all $t \in (t_1, t_2)$
  \begin{equation}
    \label{eq:coer-mod-close}
    \|\bs u(t) - \bs W_{\wt\lambda(t)}\|_\cE \leq \delta,\qquad \wt\lambda(t) > 0,
  \end{equation}
  then there exist unique functions $\lambda(t) \in C^1((t_1, t_2), (0, +\infty))$ and $a(t) \in C^1((t_1, t_2), \bR)$ such that for
  \begin{equation}
    \label{eq:coer-g}
    \bs g(t) := \bs u(t) - \bs U_{\lambda(t)}^{a(t)}
  \end{equation}
  the following holds for all $t \in (t_1, t_2)$:
  \begin{align}
    \la\cZ_\uln{\lambda(t)}, g(t)\ra = \la \alpha_{\lambda(t)}^-, \bs g(t)\ra &= 0, \label{eq:coer-gorth} \\
    \|\bs g(t)\|_\cE &\leq \eta, \label{eq:coer-gbound} \\
    |\lambda(t)/\wt \lambda(t) - 1|  + |a(t)| &\leq \eta. \label{eq:coer-lambda-range}
  \end{align}
  In addition,
  \begin{align}
    |\lambda'(t)| &\lesssim \|\bs g(t)\|_\cE, \label{eq:coer-mod-l} \\
    \big|a'(t) + \frac{\nu}{\lambda(t)}a(t)\big| &\lesssim \frac{1}{\lambda(t)}\big(|a(t)|\cdot\|\bs g(t)\|_\cE + \|\bs g(t)\|_\cE^2\big). \label{eq:coer-mod-a}
  \end{align}
\end{lemma}
\begin{proof}
  We follow a standard procedure, see for instance \cite[Proposition 1]{MaMe01}.
  \paragraph{\textbf{Step 1.}}
  We will first show that for fixed $t_0 \in (t_1, t_2)$ there exist unique $\lambda(t_0)$ and $a(t_0)$ such that
  \eqref{eq:coer-gorth}, \eqref{eq:coer-gbound} and \eqref{eq:coer-lambda-range} hold for $t = t_0$. Without loss of generality we can assume that $\wt \lambda(t_0) = 1$
  (it suffices to rescale everything).

  We consider $\Phi: \cE\times \bR^2 \to \bR^2$ defined as
  \begin{equation*}
    \begin{aligned}
    \Phi(\bs u_0; l_0, a_0) &= \big(\Phi_1(\bs u_0; l_0, a_0), \Phi_2(\bs u_0; l_0, a_0)\big) \\
    &:= \big(\la \eee^{-l_0}\cZ_\uln{\eee^{l_0}}, u_0 - U_{\eee^{l_0}}^{a_0}\ra, \la\alpha_{\eee^{l_0}}^-, \bs u_0 - \bs U_{\eee^{l_0}}^{a_0}\ra\big).
  \end{aligned}
  \end{equation*}
  One easily computes:
  \begin{align}
    \partial_{l}\Phi_1(\bs W; 0, 0) &= \la \cZ, \Lambda W\ra > 0, \\
    \partial_{l}\Phi_2(\bs W; 0, 0) &= 0, \\
    \partial_{a}\Phi_1(\bs W; 0, 0) &= 0, \\
    \partial_{a}\Phi_2(\bs W; 0, 0) &= -\la \alpha^-, \cY^-\ra = -1. \nonumber
  \end{align}
  Applying the Implicit Function Theorem with $\bs u_0 := \bs u(t_0)$ we obtain existence of parameters $a_0 =: a(t_0)$ and $\lambda_0 = \eee^{l_0} =: \lambda(t_0)$.
  \paragraph{\textbf{Step 2.}}
  We will show that $\lambda(t)$ (equivalently, $l(t) := \log(\lambda(t))$) and $a(t)$ are $C^1$ functions of $t$.

  Take $t_0 \in (t_1, t_2)$ and let $a_0 := a(t_0)$, $l_0 := \log(\lambda(t_0))$.
  Define $(\wt l, \wt a): (t_0-\varepsilon, t_0+\varepsilon) \to \bR^2$ as the solution of the differential equation
  \begin{equation}
    \label{eq:lambda-der}
    \dd t(\wt l(t), \wt a(t)) = -(\partial_{l, a} \Phi)^{-1}(\vD_{\bs u}\Phi)\partial_t \bs u(t)
  \end{equation}
  with the initial condition $\wt l(t_0) = l_0$, $\wt a(t_0) = a_0$.
  Notice that $\vD_{\bs v}\Phi$ is a continuous functional on $\cF$, so we can apply it to $\partial_t \bs u(t)$.

  Using the chain rule we get $\dd t \Phi(\bs u(t); \wt l(t), \wt a(t)) = 0$ for $t \in (t_0 - \varepsilon, t_0 + \varepsilon)$.
  By continuity, $|\wt l(t) - l_0| <\eta$ in some neighbourhood of $t = t_0$.
  Hence, by the uniqueness part of the Implicit Function Theorem, we get $\wt l(t) = \log \lambda(t)$ and $\wt a(t) = a(t)$ in some neighbourhood of $t = t_0$.
  In particular, $\lambda(t)$ and $a(t)$ are of class $C^1$ in some neighbourhood of $t_0$.
  \paragraph{\textbf{Step 3.}}
  From \eqref{eq:coer-g} we obtain the following differential equation of the error term $\bs g$:
\begin{equation*}
  \partial_t \bs g = \partial_t(\bs u - \bs U_\lambda^a) = J\circ\big(\vD E(\bs u) - \vD E(\bs U_\lambda^a)\big) - \big(\partial_t \bs U_\lambda^a - J\circ \vD E(\bs U_\lambda^a)\big).
\end{equation*}
We have
\begin{equation}
  \label{eq:deriv-t-U}
  \partial_t \bs U_\lambda^a = \lambda'\partial_\lambda \bs U_\lambda^a + a'\partial_a \bs U_\lambda^a = -\lambda'\cdot \frac{1}{\lambda}\Lambda_\cE\bs U_\lambda^a +a' \partial_a \bs U_\lambda^a,
\end{equation}
so using \eqref{eq:coer-daU} we get
\begin{equation}
  \label{eq:coer-mod-g}
  \partial_t \bs g = J\circ(\vD E(\bs U_\lambda^a + \bs g) - \vD E(\bs U_\lambda^a)) + \lambda'\cdot \frac{1}{\lambda}\Lambda_\cE\bs U_\lambda^a - \Big(a'+\frac{\nu a}{\lambda}\Big)\partial_a \bs U_\lambda^a.
\end{equation}
The first component reads:
\begin{equation*}
  \partial_t g = \dot g + \lambda'\Lambda U_\uln\lambda^a + \big(1+\frac{\lambda a'}{\nu a}\big)\dot U_\uln\lambda^a,
\end{equation*}
hence differentiating in time the first orthogonality relation $\la \frac{1}{\lambda}\cZ_\uln\lambda, g\ra = 0$ we obtain
  \begin{equation}
    \label{eq:mod-orth-1}
    0 = \dd t\la \frac{1}{\lambda}\cZ_\uln\lambda, g\ra = -\frac{\lambda'}{\lambda^2}\la \Lambda_0\cZ_\uln\lambda, g\ra + \frac{1}{\lambda}\la \cZ_\uln\lambda, \dot g\ra + \frac{\lambda'}{\lambda^2}\la\cZ_\uln\lambda, \Lambda U_\lambda^a\ra + \Big(\frac{1}{\lambda} + \frac{a'}{\nu a}\Big)\la \cZ_\uln\lambda, \dot U_\uln\lambda^a\ra.
  \end{equation}
  Differentiating the second orthogonality relation $\la \alpha_\lambda^-, \bs g\ra = 0$ and using \eqref{eq:coer-mod-g} we obtain
  \begin{equation*}
    \begin{aligned}
      0 = \dd t\la\alpha_\lambda^-, \bs g\ra &= -\frac{\lambda'}{\lambda}\la\Lambda_{\cE^*}\alpha_\lambda^-, \bs g\ra +
      \la\alpha_\lambda^-, J\circ(\vD E(\bs U_\lambda^a + \bs g) - \vD E(\bs U_\lambda^a))\ra \\
    &+ \frac{\lambda'}{\lambda}\la\alpha_\lambda^-, \Lambda_\cE\bs U_\lambda^a\ra - \Big(a' + \frac{\nu a}{\lambda}\Big)\big\la\alpha_\lambda^-, \partial_a \bs U_\lambda^a\big\ra.
  \end{aligned}
  \end{equation*}
  Together with \eqref{eq:mod-orth-1} this yields the following linear system for $\lambda'$ and $\lambda\Big(a' + \frac{\nu a}{\lambda}\Big)$:
  \begin{equation}
    \label{eq:coer-linsys}
    \begin{pmatrix}M_{11} & M_{12} \\ M_{21} & M_{22}\end{pmatrix}\begin{pmatrix}\lambda' \\ \lambda\big(a'+\frac{\nu a}{\lambda}\big)\end{pmatrix} = 
      \begin{pmatrix}-\la\cZ_\uln\lambda, \dot h\ra \\ -\lambda\la\alpha_\lambda^-, J\circ (\vD E(\bs U_\lambda^a + \bs g) - \vD E(\bs U_\lambda^a))\ra \end{pmatrix},
  \end{equation}
  where
  \begin{equation}
    \label{eq:coer-coef}
    \begin{aligned}
      M_{11} &= \frac{1}{\lambda}\la\cZ_\uln\lambda, \Lambda U_\lambda^a\ra - \frac{1}{\lambda}\la\Lambda_{-1}\cZ_\uln\lambda, g\ra, \\
      M_{12} &= \frac{1}{\nu a}\la\cZ_\uln\lambda, \dot U_\uln\lambda^a\ra,\\
      M_{21} &=- \la\Lambda_{\cE^*}\alpha_\lambda^-, \bs g\ra+ \la\alpha_\lambda^-, \Lambda_\cE\bs U_\lambda^a\ra, \\
      M_{22} &= -\la\alpha_\lambda^-, \partial_a \bs U_\lambda^a\ra.
    \end{aligned}
  \end{equation}
  Since $\frac{1}{\lambda}\la \cZ_\uln\lambda, \Lambda W_\lambda\ra \gtrsim 1$, $\la \alpha_{\lambda}^-, \Lambda_\cE \bs W\ra = 0$,
  $\la \alpha_{\lambda}^-, \cY_\lambda^-\ra = 1$, $\|\Lambda_\cE \bs U_\lambda^a - \Lambda_\cE \bs W_\lambda\|_\cE \lesssim |a|$
  and $\|\partial_a \bs U_\lambda^a - \cY_\lambda^-\|_\cE \lesssim |a|$, we see that
  $$
  \begin{aligned}
  &|M_{11}| \sim &1,\quad &|M_{12}| \lesssim &1, \\
  &|M_{21}| \lesssim &\|\bs g\|_\cE + |a|,\quad &|M_{22}| \sim &1.
\end{aligned}
$$
  Hence, $\Big|\det \begin{pmatrix}M_{11} & M_{12} \\ M_{21} & M_{22}\end{pmatrix}\Big| \gtrsim 1$ and we obtain
    \begin{equation}
      \label{eq:coer-matrix-inverse}
      \begin{aligned}
        |\lambda'| &\lesssim |M_{22}|\cdot|\la\cZ_\uln\lambda, \dot g\ra| + |M_{12}|\cdot|\lambda\la\alpha_\lambda^-, J\circ (\vD E(\bs U_\lambda^a + \bs g) - \vD E(\bs U_\lambda^a))\ra|, \\
        \big|a + \frac{\lambda a'}{\nu}\big| &\lesssim |M_{21}|\cdot|\la\cZ_\uln\lambda, \dot g\ra| + |M_{11}|\cdot|\lambda\la\alpha_\lambda^-, J\circ (\vD E(\bs U_\lambda^a + \bs g) - \vD E(\bs U_\lambda^a))\ra|.
      \end{aligned}
    \end{equation}
  Since $\la\alpha_\lambda^-, J\circ\vD^2 E(\bs W_\lambda)\bs g\ra = -\frac{\nu}{\lambda}\la\alpha_\lambda^-, \bs g\ra = 0$,
  Lemma~\ref{lem:weak-linea} implies that
  \begin{equation}
    \label{eq:coer-alpha-miracle}
      |\la\alpha_\lambda^-, J\circ(\vD E(\bs U_\lambda^a + \bs g) - \vD E(\bs U_\lambda^a))\ra| \lesssim \frac{1}{\lambda}\|\bs g\|_\cE\cdot(|a| + \|\bs g\|_\cE).
  \end{equation}
  Now \eqref{eq:coer-mod-l} and \eqref{eq:coer-mod-a} follow from \eqref{eq:coer-coef}, \eqref{eq:coer-matrix-inverse} and \eqref{eq:coer-alpha-miracle}.

\end{proof}

In the rest of this section $\lambda(t)$ and $a(t)$ denote the modulation parameters obtained in Lemma~\ref{lem:coer-mod} and $\bs g(t)$ is the function defined by \eqref{eq:coer-g}.

For given modulation parameters $\lambda$ and $a$ we define:
\begin{equation}
  \label{eq:beta}
  \beta_\lambda^a := -\frac{\nu}{2\lambda}J\partial_a \bs U^a_\lambda.
\end{equation}
We see that $\beta_\lambda^0 = -\frac{\nu}{2\lambda}J\cY_\lambda^- = \alpha_\lambda^+$,
and indeed it turns out that $\beta_\lambda^a$ is a refined version of $\alpha_\lambda^+$,
adapted to the situation when $|a| \gg \|\bs h\|_\cE$.

\begin{proposition}
  \label{prop:coer-dtb}
  The function
  $$ b(t):= \la \beta_{\lambda(t)}^{a(t)}, \bs g(t)\ra $$
  satisfies
  \begin{equation*}
    \big|\dd t b(t) - \frac{\nu}{\lambda(t)}b(t)\big| \lesssim \frac{1}{\lambda(t)}\cdot\|\bs g(t)\|_\cE^2.
  \end{equation*}
\end{proposition}
\begin{proof}~
  \paragraph{\textbf{Step 1.}}
  We check that
  \begin{align}
    |\la \beta_\lambda^a - \alpha_\lambda^+, \bs g\ra| &\lesssim |a|\cdot\|\bs g\|_\cE, \label{eq:beta-alpha-1} \\
    |\la \beta_\lambda^a - \alpha_\lambda^+, \partial_t \bs g\ra| &\lesssim \frac{1}{\lambda}|a|\cdot \|\bs g\|_\cE, \label{eq:beta-alpha-2} \\
    |\la \partial_a \beta_\lambda^a, \bs g\ra| + |\lambda\partial_\lambda\beta_\lambda^a, \bs g\ra| &\lesssim \|\bs g\|_\cE. \label{eq:deriv-beta}
  \end{align}
From Proposition~\ref{prop:Ua} we have $\|\beta_1^a - \alpha^+\|_{Y^{k}\times Y^{k+1}} \lesssim |a|$,
and \eqref{eq:beta-alpha-1} follows by rescaling.

Similarly one obtains
\begin{align}
  |\la \beta_\lambda^a - \alpha_\lambda^+, J\circ(\vD E(\bs U_\lambda^a + \bs g) - \vD E(\bs U_\lambda^a))\ra| &\lesssim \frac{1}{\lambda}|a|\cdot\|\bs g\|_\cE, \label{eq:coer-beta-alpha-4} \\
  |\la \beta_\lambda^a - \alpha_\lambda^+, \Lambda_\cE\bs U_\lambda^a\ra| +|\la \beta_\lambda^a - \alpha_\lambda^+, \partial_a\bs U_\lambda^a\ra|&\lesssim |a| \label{eq:coer-beta-alpha-5},
\end{align}
hence \eqref{eq:beta-alpha-2} follows from \eqref{eq:coer-mod-l} and \eqref{eq:coer-mod-a}.

Note that \eqref{eq:beta-alpha-1} implies in particular that $|\la \beta_\lambda^a, \bs g\ra| \lesssim \|\bs g\|_\cE$ with a universal constant.

\paragraph{\textbf{Step 2.}}
Consider the case
\begin{equation}
  \label{eq:coer-a-leq-g}
|a(t)| \leq \|\bs g(t)\|_\cE.
\end{equation}
We have
\begin{equation}
  \label{eq:coer-dtb}
  \dd t b(t) = \la \beta_{\lambda(t)}^{a(t)}, \partial_t \bs g(t)\ra + \lambda'(t)\la \partial_\lambda \beta_{\lambda(t)}^{a(t)}, \bs g(t)\ra + a'(t)\la \partial_a \beta_{\lambda(t)}^{a(t)}, \bs g\ra.
\end{equation}
From Lemma~\ref{lem:coer-mod} we know that $|\lambda'| \lesssim \|\bs g\|_\cE$ and $|a'| \lesssim \frac{1}{\lambda}\|\bs g\|_\cE$. Hence from \eqref{eq:deriv-beta}
it follows that the last two terms of \eqref{eq:coer-dtb} are negligible.

Using \eqref{eq:coer-a-leq-g}, \eqref{eq:beta-alpha-1} and \eqref{eq:beta-alpha-2} we see that it is sufficient to show that
\begin{equation}
  \label{eq:coer-dtalpha}
  \big|\la \alpha_\lambda^+, \partial_t \bs g\ra - \frac{\nu}{\lambda}\la \alpha_\lambda^+, \bs g\ra\big| = |\la \alpha_\lambda^+, \partial_t \bs g - J\circ\vD^2 E(\bs W_\lambda)\bs g\ra| \lesssim \frac{1}{\lambda}\|\bs g\|^2.
\end{equation}
This follows easily from \eqref{eq:coer-mod-g}. Indeed, from Lemma~\ref{lem:weak-linea} we deduce that
$$
|\la \alpha_\lambda^+, J\circ(\vD E(\bs U^a_\lambda + \bs g) - \vD E(\bs U^a_\lambda) - \vD^2 E(\bs W_\lambda)\bs g)\ra| \lesssim \frac{1}{\lambda}\|\bs g\|_\cE^2.
$$
To see that the contribution of the last two terms in \eqref{eq:coer-mod-g} is negligible, it suffices to use \eqref{eq:coer-mod-l}, \eqref{eq:coer-mod-a},
$|\la \alpha_\lambda^+, \Lambda_\cE \bs U_\lambda^a\ra| \lesssim |a|$ and $|\la \alpha_\lambda^+, \partial_a \bs U_\lambda^a\ra| \lesssim 1$.

\paragraph{\textbf{Step 3.}}
Now consider the case
\begin{equation}
  \label{eq:coer-g-leq-a}
  \|\bs g(t)\|_\cE \leq |a(t)|.
\end{equation}
We can assume that $a \neq 0$ (otherwise $\bs u(t) \equiv \bs W_\lambda$ and the conclusion is obvious).

Using Proposition~\ref{prop:Ua} we get
\begin{equation}
  \label{eq:coer-beta-form-2}
\beta_\lambda^a = -\frac{1}{2a}\vD E(\bs U_\lambda^a) \quad \Rightarrow \quad b(t) = -\frac{1}{2a(t)}\cdot \la \vD E(\bs U_{\lambda(t)}^{a(t)}), \bs g(t)\ra.
\end{equation}
The idea of the proof is that the first factor grows exponentially, while the second does not change much.
From \eqref{eq:coer-mod-a} and \eqref{eq:coer-g-leq-a} we obtain $\big|\frac{a'(t)}{a(t)} + \frac{\nu}{\lambda(t)}\big| \lesssim \frac{1}{\lambda(t)}\|\bs g(t)\|$, hence
\begin{equation*}
  \begin{aligned}
  \dd t b(t) &= -\frac{a'(t)}{a(t)}b(t) -\frac{1}{2a(t)}\dd t\big\la\vD E(\bs U_{\lambda(t)}^{a(t)}),\bs g(t)\big\ra \\ 
  &= \frac{\nu}{\lambda(t)}b(t) -\frac{1}{2a(t)}\dd t\big\la \vD E(\bs U_{\lambda(t)}^{a(t)}),\bs g(t)\big\ra + \frac{1}{\lambda(t)}O(\|\bs g(t)\|_\cE^2).
\end{aligned}
  \end{equation*}
  We compute the second term using \eqref{eq:coer-mod-g}:
  \begin{equation*}
    \begin{aligned}
    \dd t\la \vD E(\bs U_\lambda^a), \bs g\ra  &= \la \vD^2 E(\bs U_\lambda^a)\partial_t \bs U_\lambda^a, \bs g\ra  \\ &
    + \la \vD E(\bs U_\lambda^a), J\circ(\vD E(\bs U_\lambda^a + \bs g) - \vD E(\bs U_\lambda^a)) + \lambda'\cdot \frac{1}{\lambda}\Lambda_\cE\bs U_\lambda^a - \Big(a'+\frac{\nu a}{\lambda}\Big)\partial_a \bs U_\lambda^a\ra.
  \end{aligned}
  \end{equation*}
  Observe that
  \begin{equation}
    \label{eq:Ua-crit}
    \begin{aligned}
  \la \vD E(\bs U_\lambda^a), \Lambda_\cE \bs U_\lambda^a) = -\partial_\lambda E(\bs U_\lambda^a) = 0, \\
  \la \vD E(\bs U_\lambda^a), \partial_a \bs U_\lambda^a) = \partial_a E(\bs U_\lambda^a) = 0.
\end{aligned}
\end{equation}
Since $\vD E(\bs U_\lambda^a) \in Y^k\times Y^{k+1}$ by Proposition~\ref{prop:Ua}, Lemma~\ref{lem:weak-linea} implies that
$$
\big|\big\la \vD E(\bs U_\lambda^a), J\circ\big(\vD E(\bs U_\lambda^a + \bs g) - \vD E(\bs U_\lambda^a) - \vD^2 E(\bs U_\lambda^a)\bs g\big)\big\ra\big| \lesssim \frac{a}{\lambda}\|\bs g\|_\cE^2,
$$
hence using self-adjointness of $\vD^2 E(\bs U_\lambda^a)$ and anti-self-adjointness of $J$ we get
\begin{equation*}
 \dd t\la \vD E(\bs U_\lambda^a), \bs g\ra = \la \vD^2 E(\bs U_\lambda^a)\big(\partial_t \bs U_\lambda^a - J\circ\vD E(\bs U_\lambda^a)\big), \bs g\ra + \frac{a}{\lambda}O(\|\bs g\|_\cE^2).
\end{equation*}
The following estimates hold:
  \begin{equation}
    \label{eq:magique-U}
    \begin{aligned}
      \|\vD^2 E(\bs U_\lambda^a)\Lambda_\cE\bs U_\lambda^a\|_{\cE^*}&\lesssim |a|, \\
      \|\vD^2 E(\bs U_\lambda^a)\partial_a \bs U_{\lambda}^a\|_{\cE^*} &\lesssim 1
    \end{aligned}
  \end{equation}
  (the first one follows from $\vD^2 E(\bs W_\lambda)\Lambda_\cE\bs W_\lambda = 0$).
  Using \eqref{eq:deriv-t-U} and \eqref{eq:magique-U} together with \eqref{eq:coer-mod-l} and \eqref{eq:coer-mod-a} concludes the proof.
\end{proof}

%
%
As an application of the preceding proposition, we now show that the stable manifold $\bs U_\lambda^a$ is the only source of the lack of coercivity of the energy functional restricted to the trajectories staying close to the family of stationary states.

Given $\bs u_0 \in \cE$, let $\bs u(t): [0, T_+) \to \cE$ denote the maximal solution of \eqref{eq:nlw} with initial data $\bs u(0) = \bs u_0$.
For $\eta > 0$ sufficiently small we define the \emph{centre-stable set} $\cM_\tx{cs}$ as
\begin{equation*}
    \cM_\tx{cs} := \big\{\bs u_0: \sup_{0 \leq t < T_+}\inf_{\lambda > 0}\|\bs u(t) - \bs W_\lambda\|_\cE \leq \eta\big\}.
\end{equation*}
\begin{remark}
In the case $N = 3$ it was proved by Krieger, Nakanishi and Schlag \cite{KrNaSc15}
that $\cM_\tx{cs}$ is a local $C^1$ manifold tangent at $\bs u_0 = \bs W$ to $\cX_\tx{cs}$.
\end{remark}
\begin{remark}
  \label{rem:coer-dtb}
  It is not difficult to see that if $\cM_\tx{cs}$ is a regular hypersurface,
  then necessarily its tangent space at $\bs U_\lambda^a$ is given by
  $$
  \bs U_\lambda^a + \ker \beta_\lambda^a = \{\bs U_\lambda^a + \bs g: \la \beta_\lambda^a, \bs g\ra = 0\}.
  $$
  Hence $b(t)$ is a natural candidate to measure how a trajectory moves away from $\cM_\tx{cs}$.
\end{remark}

\begin{corollary}
  \label{cor:coer-coercivity}
    If $\eta > 0$ is small enough, then there exists a constant $C_E > 0$ such that
  \begin{equation}
    \label{eq:coercivity}
    \bs u_0 \in \cM_\tx{cs}\quad\Rightarrow\quad \inf_{\lambda > 0, a \in \bR}\|\bs u_0 - \bs U_\lambda^a\|_\cE^2 \leq C_E\big(E(\bs u_0) - E(\bs W)\big).
  \end{equation}
\end{corollary}

\begin{proof}~
  \paragraph{\textbf{Step 1 -- Coercivity.}} We will prove that if $\|\bs g\|_\cE$ is small enough and $\la \cZ_{\uln\lambda}, g\ra = \la \alpha_\lambda^-, \bs g\ra = 0$, then
  \begin{equation}
    \label{eq:coer-coer}
    E(\bs U_\lambda^a + \bs g) - E(\bs W) + 2a\la\beta_\lambda^a, \bs g\ra + 2|\la\beta_\lambda^a, \bs g\ra|^2 \sim \|\bs g\|_\cE^2.
  \end{equation}
  We have $2a \la \beta_\lambda^a, \bs g\ra = \la \vD E(\bs U_\lambda^a), \bs g\ra$, hence Lemma~\ref{lem:taylor} implies
  \begin{equation*}
    \begin{aligned}
    E(\bs U_\lambda^a + \bs g) &= E(\bs U_\lambda^a) + \la \vD E(\bs U_\lambda^a), \bs g\ra + \frac 12\la \vD^2 E(\bs U_\lambda^a)\bs g, \bs g\ra + o(\|\bs g\|_\cE^2) \\
    &=E(\bs W) - 2a\la\beta_\lambda^a, \bs g\ra + \frac 12\la \vD^2 E(\bs U_\lambda^a)\bs g, \bs g\ra + o(\|\bs g\|_\cE^2).
  \end{aligned}
  \end{equation*}
  By \eqref{eq:beta-alpha-1} we have $|\la \beta_\lambda^a, \bs g\ra^2 - \la \alpha_\lambda^+, \bs g\ra^2| \lesssim |a|\cdot\|\bs g\|^2$,
  hence Lemma~\ref{lem:coer} yields
\begin{equation*}
  \frac 12 \la \vD^2 E(\bs U_\lambda^a)\bs g, \bs g\ra + 2|\la\beta_\lambda^a, \bs g\ra|^2 \sim \|\bs g\|_\cE^2,
\end{equation*}
which implies \eqref{eq:coer-coer}.

  \paragraph{\textbf{Step 2 -- Differential inequalities.}}
  Let $\bs g(t)$, $\lambda(t)$ and $a(t)$ be given by Lemma~\ref{lem:coer-mod}. Observe that
\begin{equation}
  \label{eq:lambda-diverge}
  \int_{0}^{T_+}\frac{1}{\lambda(t)}\ud t = +\infty.
\end{equation}
Indeed, if $|\log \lambda(t)|$ is unbounded, then
\begin{equation*}
  \int_{0}^{T_+}\frac{1}{\lambda(t)}\ud t \gtrsim \int_{0}^{T_+}\frac{\|\bs g(t)\|_{\cE}}{\lambda(t)}\ud t
\gtrsim\int_{0}^{T_+}\frac{|\lambda'(t)|}{\lambda(t)}\ud t = +\infty.
\end{equation*}
If $|\log\lambda(t)|$ is bounded, then by the Cauchy theory $T_+ = +\infty$ and \eqref{eq:lambda-diverge} follows.

    From Proposition~\ref{prop:coer-dtb} it follows that there exists a constant $C_1$ such that
  \begin{equation}
    \label{eq:coer-b-destab}
    |b(t)| \geq C_1 \|\bs g(t)\|_\cE^2 \quad\Rightarrow\quad \dd t|b(t)| \geq \frac{\nu}{2\lambda(t)}|b(t)|,\qquad \forall t\in [0, T_+).
  \end{equation}
  We will show that there exists a constant $C_2$ such that
  \begin{equation}
    \label{eq:coer-bbound-boot}
    |b(t)| \geq C_2 \big(E(\bs u_0) - E(\bs W)\big)\quad \Rightarrow\quad |b(t)| \geq C_1 \|\bs g(t)\|_\cE^2.
  \end{equation}
  Indeed, we can rewrite \eqref{eq:coer-coer} as
  \begin{equation}
    \label{eq:coer-coer-2}
    E(\bs u_0) - E(\bs W) + 2a(t)b(t) + 2b(t)^2 \sim \|\bs g\|_\cE^2,
  \end{equation}
  hence if $|b(t)| \geq C_2$, then
  \begin{equation*}
    |b(t)|\cdot\big(\frac{1}{C_2} + 2|a(t)| + 2|b(t)|\big) \geq E(\bs u_0) - E(\bs W) + 2a(t)b(t) + 2b(t)^2 \gtrsim \|\bs g\|_\cE^2,
  \end{equation*}
  which implies \eqref{eq:coer-bbound-boot} since $|a(t)|$ and $|b(t)|$ are small.
  
  Suppose for the sake of contradiction that $b(0) \neq 0$ and $|b(0)| \geq 2C_2 \big(E(\bs u_0) - E(\bs W)\big)$.
  Let $t_1 \leq T_+$ be maximal such that
  \begin{equation}
    \label{eq:coer-b-big}
    b(t) \neq 0,\ |b(t)| \geq C_2 \big(E(\bs u_0) - E(\bs W)\big),\qquad \forall t \in [0, t_1).
  \end{equation}
  Of course $t_1 > 0$. Suppose that $t_1 < T_+$.
    But \eqref{eq:coer-bbound-boot} and \eqref{eq:coer-b-destab} imply that $\dd t |b(t)| > 0$ for $t \in [0, t_1]$.
    In particular, \eqref{eq:coer-b-big} cannot break down at $t = t_1$.
Thus $t_1 = T_+$ and \eqref{eq:coer-bbound-boot} implies that for $t \in [0, T_+)$ there holds $|b(t)| \geq C_1 \|\bs g\|_\cE$.
      By \eqref{eq:coer-b-destab} and \eqref{eq:lambda-diverge}, this would imply $|\beta(t)|\xrightarrow[t\to T_+]{}+\infty$, which is absurd.

      As a result, $|b(0)| \leq 2C_2 \big(E(\bs u_0) - E(\bs W)\big)$. Since $|a(0)|$ and $\|\bs g(0)\|_\cE$ may be assumed as small as we wish, the conclusion follows from \eqref{eq:coer-coer-2} applied at $t = 0$.
    \end{proof}
    \begin{remark}
    It follows quite easily from Lemma~\ref{lem:coer} that
\begin{align}
  \label{eq:lin-coer-1}
  \bs g \in \cX_\tx{c} &\Rightarrow\quad \exists b\in\bR\ :\ \|\bs g - b\Lambda_\cE \bs W\|_\cE^2 \lesssim \frac 12\la\vD^2 E(\bs W)\bs g, \bs g\ra, \\
  \label{eq:lin-coer-2}
  \bs g \in \cX_\tx{cs}&\Rightarrow\quad \exists a, b\in \bR\ :\ \|\bs g - b\Lambda_\cE \bs W - a\ym\|_\cE^2 \lesssim \frac 12\la\vD^2 E(\bs W)\bs g, \bs g\ra.
\end{align}
Corollary~\ref{cor:coer-coercivity} provides a nonlinear version of \eqref{eq:lin-coer-2}.
By similar methods (analyzing just the \emph{linear} stability and instability components $\alpha_\lambda^+$ and $\alpha_\lambda^-$)
one can prove a nonlinear analogue of \eqref{eq:lin-coer-1}, that is
  \begin{equation}
    \label{eq:coercivity-c}
    \bs u_0 \in \cM_\tx{c}\quad\Rightarrow\quad \inf_{\lambda > 0}\|\bs u_0 - \bs W_\lambda\|_\cE^2 \leq C_E\big(E(\bs u_0) - E(\bs W)\big),
  \end{equation}
where
\begin{equation*}
    \cM_\tx{c} := \cM_\tx{cs} \cap \cM_\tx{cu} = \big\{\bs u_0: \sup_{T_- < t < T_+}\inf_{\lambda > 0}\|\bs u(t) - \bs W_\lambda\|_\cE \lambda \eta\big\}.
  \end{equation*}
\end{remark}

\section{Nonexistence of pure two-bubbles with opposite signs}
\label{sec:bulles}
\subsection{Modulation near the sum of two bubbles}
\label{ssec:bulles-mod}
Because of a slow decay of $W$, we will introduce compactly supported approximations of $W_\lambda$.
Let $R > 0$ be a large constant to be chosen later.

We denote
\begin{equation}
  \label{eq:bulles-V}
  V_R(\lambda_1, \lambda_2)(x) :=
  \begin{cases}
    W_{\lambda_1}(x) - \zeta(\lambda_1, \lambda_2)\qquad &\text{when }|x| \leq R\sqrt{\lambda_1\lambda_2}, \\
    0 \qquad &\text{when }|x| \geq R\sqrt{\lambda_1\lambda_2},
  \end{cases}
\end{equation}
where
\begin{equation}
  \label{eq:bulles-zeta}
  \zeta(\lambda_1, \lambda_2) := W_{\lambda_1}(R\sqrt{\lambda_1\lambda_2}) = \frac{1}{{\lambda_1}^{\frac{N-2}{2}}}\Bigl(1+\frac{R^2\lambda_2}{N(N-2)\lambda_1}\Bigr)^{-\frac{N-2}{2}}
  = \Bigl(\lambda_1+\frac{R^2\lambda_2}{N(N-2)}\Big)^{-\frac{N-2}{2}}.
\end{equation}
We have $\zeta(\lambda_1, \lambda_2) \sim R^{-(N-2)}\lambda_2^{-\frac{N-2}{2}}$, $\partial_{\lambda_1}\zeta(\lambda_1, \lambda_2) \sim R^{-N}\lambda_2^{-\frac{N}{2}}$
and $\partial_{\lambda_2}\zeta(\lambda_1, \lambda_2) \sim R^{-N}\lambda_2^{-\frac{N}{2}}$.

We will also denote
\begin{equation*}
  \bs V_R(\lambda_1, \lambda_2) := (V_R(\lambda_1, \lambda_2), 0)\in \cE.
\end{equation*}

It is straightforward to check that $V_R(\lambda_1, \lambda_2)$ has weak derivatives $\partial_{\lambda_1}V_R(\lambda_1, \lambda_2)$ and $\partial_{\lambda_2}V_R(\lambda_1, \lambda_2)$,
which are given by the formulas:
\begin{align}
  \label{eq:bulles-Vl1}
  \partial_{\lambda_1} V_R(\lambda_1, \lambda_2)(x) &=
  \begin{cases}
    -(\Lambda W)_{\uln{\lambda_1}}(x) - \partial_{\lambda_1}\zeta(\lambda_1, \lambda_2)\qquad &\text{when }|x| < R\sqrt{\lambda_1\lambda_2}, \\
    0 \qquad &\text{when }|x| > R\sqrt{\lambda_1\lambda_2},
  \end{cases} \\
  \label{eq:bulles-Vl2}
  \partial_{\lambda_2} V_R(\lambda_1, \lambda_2)(x) &=
  \begin{cases}
    -\partial_{\lambda_2}\zeta(\lambda_1, \lambda_2)\qquad &\text{when }|x| < R\sqrt{\lambda_1\lambda_2}, \\
    0 \qquad &\text{when }|x| > R\sqrt{\lambda_1\lambda_2}.
  \end{cases}
\end{align}
Notice that $\partial_{\lambda_j} V_R(\lambda_1, \lambda_2) \in L^2$ and $\partial_{\lambda_j} V_R(\lambda_1, \lambda_2) \notin \dot H^1$.

In the whole section we will denote $\lambda := \frac{\lambda_1}{\lambda_2}$ and $\cN(\bs g, \lambda) := \sqrt{\|\bs g\|_\cE^2 + \lambda^\frac{N-2}{2}}$.
\begin{lemma}
  \label{lem:bulles-prop-V}
  For $R \gg 1$ and $\lambda \ll 1$ the following estimates are true with constants depending only on the dimension:
  \begin{align}
    \|V_R(\lambda_1, \lambda_2) - W_{\lambda_1}\|_{\dot H^1} &\lesssim R^{-\frac{N-2}{2}}\lambda^{\frac{N-2}{4}}, \label{eq:db-V-W} \\
    \|V_R(\lambda_1, \lambda_2) - W_{\lambda_1}\|_{L^\infty} &\lesssim R^{-N+2}\lambda_2^{-\frac{N-2}{2}}, \label{eq:db-V-Linf} \\
    \|\partial_{\lambda_1} V_R(\lambda_1, \lambda_2) + \Lambda W_\uln{\lambda_1}\|_{L^\infty(|x| < R\sqrt{\lambda_1\lambda_2})} &\lesssim R^{-N}\lambda_2^{-\frac{N}{2}}, \label{eq:db-lV-Linf} \\
    \|V_R(\lambda_1, \lambda_2)\|_{L^1} &\lesssim R^2 \lambda_2^\frac{N+2}{2}\lambda^\frac N2, \label{eq:db-V-L1} \\
    \|\partial_{\lambda_1} V_R(\lambda_1, \lambda_2)\|_{L^1} &\lesssim R^2 \lambda_2^\frac N2 \lambda^\frac{N-2}{2}. \label{eq:db-lV-L1}
  \end{align}
\end{lemma}
\begin{proof}
  The proof of \eqref{eq:db-V-W}, \eqref{eq:db-V-Linf} and \eqref{eq:db-lV-Linf} is straightforward, see \cite[Lemma 2.3]{moi15p-2};
  \eqref{eq:db-V-L1} and \eqref{eq:db-lV-L1} follow from the fact that
  $|V_R(x)| \lesssim \lambda_1^{-\frac{N-2}{2}}\cdot \big(\frac{|x|}{\lambda_1}\big)^{-N +2}$,
  $|\partial_{\lambda_1} V_R(x)| \lesssim \lambda_1^{-\frac N2}\cdot \big(\frac{|x|}{\lambda_1}\big)^{-N +2}$
  and $\supp \big(V(x)\big) = \supp\big(\partial_{\lambda_1} V(x)\big) = B(0, R\sqrt{\lambda_1 \lambda_2})$.
\end{proof}

We will omit the subscript and write $\bs V(\lambda_1, \lambda_2)$ instead of $\bs V_R(\lambda_1, \lambda_2)$. The approximate solution we will consider is defined as follows:
\begin{equation*}
  \bs U(\lambda_1, \lambda_2, a_2) := \bs U_{\lambda_2}^{a_2} - \bs V(\lambda_1, \lambda_2).
\end{equation*}
Observe that
\begin{align}
  \partial_{\lambda_1}\bs U(\lambda_1, \lambda_2, a_2) &= -\partial_{\lambda_1}\bs V(\lambda_1, \lambda_2), \label{eq:bulles-dUl1} \\
  \partial_{\lambda_2}\bs U(\lambda_1, \lambda_2, a_2) &= -\frac{1}{\lambda_2}\Lambda_\cE \bs U_{\lambda_2}^{a_2} -\partial_{\lambda_2}\bs V(\lambda_1, \lambda_2), \label{eq:bulles-dUl2} \\
  \partial_{a_2}\bs U(\lambda_1, \lambda_2, a_2) &= \partial_a \bs U_{\lambda_2}^{a_2} = -\frac{\lambda_2}{\nu a_2}J\circ \vD E(\bs U_{\lambda_2}^{a_2}). \label{eq:bulles-dUa}
\end{align}

\begin{remark}
  \label{rem:implicit}
  The following version of the Implicit Function Theorem has the advantage of providing a lower bound
  on the size of a ball where it can be applied:

  \textit{  Suppose that $X$, $Y$ and $Z$ are Banach spaces, $x_0 \in X$, $y_0 \in Y$, $\rho, \eta > 0$ and $\Phi: B(x_0, \rho) \times B(y_0, \eta) \to Z$
  is continuous in $x$ and continuously differentiable in $y$,
  $\Phi(x_0, y_0) = 0$ and $\vD_y \Phi(x_0, y_0) =: L_0$ has a bounded inverse.
  Suppose that
  \begin{align}
    \|L_0 - \vD_y \Phi(x, y)\|_Z \leq \frac 13 \|L_0^{-1}\|_{\scrL(Z, Y)}^{-1}\qquad &\text{for }\|x-x_0\|_X < \rho, \|y-y_0\|_Y < \eta, \label{eq:implicit-1} \\
    \|\Phi(x, y_0)\|_Z \leq \frac{\eta}{3} \|L_0^{-1}\|_{\scrL(Z, Y)}^{-1}\qquad &\text{for }\|x-x_0\|_X < \rho. \label{eq:implicit-2}
  \end{align}
Then there exists $y\in C(B(x_0, \rho), B(y_0, \eta))$ such that for $x \in B(x_0, \rho)$, $y = y(x)$ is the unique solution of the equation $\Phi(x, y(x)) = 0$ in $B(y_0, \eta)$. } \qed

  The proof is the same as standard proofs of IFT, see for instance \cite[Section 2.2]{chow-hale}.
\end{remark}
\begin{lemma}
  \label{lem:bulles-mod}
  Let $\delta_0 > 0$ and $\lambda_0 > 0$ be sufficiently small. For any $0 \leq \delta \leq \delta_0$ and $0 < \wt \lambda \leq \lambda_0$
  there exists $0 \leq \eta = \eta(\delta, \wt \lambda) \sto{\delta, \wt\lambda \to 0} 0$ such that
  if $\bs u: (t_1, t_2) \to \cE$ is a solution of \eqref{eq:nlw} satisfying for all $t \in (t_1, t_2)$
  \begin{equation}
    \label{eq:bulles-mod-close}
    \|\bs u(t) -(-\bs W_{\wt\lambda_1(t)} + \bs W_{\wt\lambda_2(t)})\|_\cE \leq \delta,\qquad 0 < \frac{\wt\lambda_1(t)}{\wt\lambda_2(t)} \leq \wt \lambda,
  \end{equation}
  then there exist unique functions $\lambda_1(t) \in C^1((t_1, t_2), (0, +\infty))$, $\lambda_2(t) \in C^1((t_1, t_2), (0, +\infty))$  and $a_2(t) \in C^1((t_1, t_2), \bR)$ such that for
  \begin{equation}
    \label{eq:bulles-g}
    \bs g(t) := \bs u(t) - \bs U(\lambda_1, \lambda_2, a_2)
  \end{equation}
  the following holds for all $t \in (t_1, t_2)$:
  \begin{align}
    \la\cZ_\uln{\lambda_1(t)}, g(t)\ra = \la\cZ_\uln{\lambda_2(t)}, g(t)\ra = \la \alpha_{\lambda_2(t)}^-, \bs g(t)\ra &= 0, \label{eq:bulles-gorth} \\
    \|\bs g(t)\|_\cE &\leq \eta, \label{eq:bulles-gbound} \\
    |\lambda_1(t)/\wt \lambda_1(t) - 1| + |\lambda_2(t)/\wt \lambda_2(t) - 1| + |a_2(t)| &\leq \eta. \label{eq:bulles-lambda-range}
  \end{align}
  In addition,
  \begin{align}
    |\lambda_1'(t)| + |\lambda_2'(t)| &\lesssim \cN(\bs g(t), \lambda(t)), \label{eq:bulles-mod} \\
    \big|a_2'(t) + \frac{\nu}{\lambda_2(t)}a_2(t)\big| &\lesssim \frac{1}{\lambda_2(t)}\big(|a_2(t)|\cdot \cN(\bs g(t), \lambda(t))+ \cN(\bs g(t), \lambda(t))^2\big), \label{eq:bulles-mod-a}
  \end{align}
  with constants which may depend on $R$.
\end{lemma}
\begin{proof}
  We will follow the same scheme as in the proof of Lemma~\ref{lem:coer-mod}.
  One additional difficulty is that we cannot reduce by rescaling to modulation near one specific function as we did before.

  \paragraph{\textbf{Step 1.}}
  We consider $\Phi: \cE \times \bR^3 \to \bR^3$ defined as
  \begin{equation*}
    \begin{aligned}
    &\Phi(\bs u_0; l_1, l_2, a_2) = \big(\Phi_1(\bs u_0; l_1, l_2, a_2), \Phi_2(\bs u_0; l_1, l_2, a_2), \Phi_3(\bs u_0; l_1, l_2, a_2)\big) \\
    := &\big(\la \frac{1}{\lambda_1}\cZ_\uln{\lambda_1}, u_0 - U(\lambda_1, \lambda_2, a_2)\ra, \la \frac{1}{\lambda_2}\cZ_\uln{\lambda_2}, u_0 - U(\lambda_1, \lambda_2, a_2)\ra, \la\alpha_{\lambda_2}^-, \bs u_0 - \bs U(\lambda_1, \lambda_2, a_2)\ra\big),
  \end{aligned}
  \end{equation*}
  where we have already written $\lambda_j$ instead of $\eee^{l_j}$ in order to simplify the notation.
  We will verify that the assumptions \eqref{eq:implicit-1} and \eqref{eq:implicit-2} are satisfied for
  $x_0 = \bs U(\wt \lambda_1, \wt \lambda_2, 0)$, $y_0 = (\wt l_1, \wt l_2, 0)$ (where $\wt l_j := \log \wt \lambda_j$),
  $\rho$ small and $\eta = C\rho$ with $C$ a universal constant. We define:
  \begin{align}
    M_{11}(\bs g; \lambda_1, \lambda_2, a_2) &:= \la\frac{1}{\lambda_1}\cZ_{\lambda_1}, \lambda_1 \partial_{\lambda_1}V(\lambda_1, \lambda_2)\ra - \la \frac{1}{\lambda_1}\Lambda_{-1}\cZ_\uln{\lambda_1}, g\ra, \\
        M_{12}(\bs g; \lambda_1, \lambda_2, a_2) &:= \la \frac{1}{\lambda_1}\cZ_\uln{\lambda_1}, \Lambda U_{\lambda_2}^{a_2} + \lambda_2 \partial_{\lambda_2}V(\lambda_1, \lambda_2)\ra , \\
        M_{13}(\bs g; \lambda_1, \lambda_2, a_2) &:= -\la \frac{1}{\lambda_1}\cZ_\uln{\lambda_1}, \partial_a U_{\lambda_2}^{a_2}\ra , \\
        M_{21}(\bs g; \lambda_1, \lambda_2, a_2) &:= \la \frac{1}{\lambda_2} \cZ_\uln{\lambda_2}, \lambda_1 \partial_{\lambda_1}V(\lambda_1, \lambda_2)\ra, \\
        M_{22}(\bs g; \lambda_1, \lambda_2, a_2) &:= \la\frac{1}{\lambda_2}\cZ_\uln{\lambda_2}, \Lambda U_{\lambda_2}^{a_2} + \lambda_2\partial_{\lambda_2}V(\lambda_1, \lambda_2) \ra - \la \frac{1}{\lambda_2}\Lambda_{-1}\cZ_\uln{\lambda_2}, g\ra, \\
        M_{23}(\bs g; \lambda_1, \lambda_2, a_2) &:= -\la \frac{1}{\lambda_2}\cZ_\uln{\lambda_2}, \partial_a U_{\lambda_2}^{a_2}\ra , \\
        M_{31}(\bs g; \lambda_1, \lambda_2, a_2) &:= \la \alpha_{\lambda_2}, \lambda_1 \partial_{\lambda_1}V(\lambda_1, \lambda_2)\ra , \\
        M_{32}(\bs g; \lambda_1, \lambda_2, a_2) &:= -\la \Lambda_{\cE^*}\alpha_{\lambda_2}, \bs g\ra + \la \alpha_{\lambda_2}, \Lambda_{\cE}\bs U_{\lambda_2}^{a_2} + \lambda_2 \partial_{\lambda_2}\bs V(\lambda_1, \lambda_2)\ra , \\
        M_{33}(\bs g; \lambda_1, \lambda_2, a_2) &:= -\la \alpha_{\lambda_2}, \partial_a \bs U_{\lambda_2}^{a_2}\ra , \\
  \end{align}
  A straightforward computation yields
  \begin{equation}
    \label{eq:bulles-triangular}
  \begin{aligned}
    &|M_{11}| \sim &1,\quad &|M_{12}| \lesssim &1,\quad &|M_{13}| \lesssim &1, \\
    &|M_{21}| \lesssim &\lambda^\frac N2,\quad &|M_{22}| \lesssim &1,\quad &|M_{23}| \lesssim &1, \\
    &|M_{31}| \lesssim &\lambda^\frac N2,\quad &|M_{32}| \lesssim &\cN(\bs g, \lambda) + |a_2|,\quad &|M_{33}| \sim &1.
  \end{aligned}
\end{equation}

  Using \eqref{eq:bulles-dUl1}, \eqref{eq:bulles-dUl2}, \eqref{eq:bulles-dUa}
  and the fact that $\partial_{l_j} = \lambda_j \partial_{\lambda_j}$ we see that
  $$
  \begin{aligned}
    M_{jk}(\bs u_0 - \bs U(\lambda_1, \lambda_2, a_2); \lambda_1, \lambda_2, a_2) &= \partial_{l_k}\Phi_j(\bs u_0; l_1, l_2, a_2), \qquad &j\in\{1, 2, 3\},&~k\in \{1, 2\}, \\
    M_{j3}(\bs u_0 - \bs U(\lambda_1, \lambda_2, a_2); \lambda_1, \lambda_2, a_2) &= \partial_{a_2} \Phi_j(\bs u_0; l_1, l_2, a_2), \qquad &j\in\{1, 2, 3\},&
  \end{aligned}
  $$
  hence \eqref{eq:bulles-triangular} implies that the jacobian matrix of $\Phi$ with respect to the modulation parameters
  is uniformly non-degenerate in a neighbourhood of $\bs U(\lambda_1, \lambda_2, a_2)$.
  This yields parameters $\lambda_1(t_0)$, $\lambda_2(t_0)$ and $a_2(t_0)$, see Remark~\ref{rem:implicit}.
  
  \paragraph{\textbf{Step 2.}}
  The argument from the proof of Lemma~\ref{lem:coer-mod} shows that $\lambda_1(t)$, $\lambda_2(t)$ and $a_2(t)$ are $C^1$ functions of $t \in (t_1, t_2)$.
  \paragraph{\textbf{Step 3.}}
  From \eqref{eq:bulles-g} we obtain the following differential equation of the error term $\bs g$:
\begin{equation*}
  \partial_t \bs g = \partial_t(\bs u - \bs U(\lambda_1, \lambda_2, a_2)) = J\circ\vD E(\bs U(\lambda_1, \lambda_2, a_2)) - \partial_t \bs U_{\lambda_2}^{a_2}.
\end{equation*}
Using \eqref{eq:bulles-dUl1}, \eqref{eq:bulles-dUl2} and \eqref{eq:bulles-dUa} this can rewritten as
\begin{equation}
  \label{eq:bulles-mod-g}
  \begin{aligned}
  \partial_t \bs g &= J\circ(\vD E(\bs U(\lambda_1, \lambda_2, a_2) + \bs g) - \vD E(\bs U_{\lambda_2}^{a_2})) \\
  &+\lambda_1' \partial_{\lambda_1}\bs V(\lambda_1, \lambda_2) +\lambda_2'\cdot\Big(\frac{1}{\lambda_2}\Lambda_\cE \bs U_{\lambda_2}^{a_2} +\partial_{\lambda_2}\bs V(\lambda_1, \lambda_2)\Big)- \big(a_2'+\frac{\nu}{\lambda_2} a_2\big)\partial_a \bs U_{\lambda_2}^{a_2}.
\end{aligned}
\end{equation}
The first component reads:
\begin{equation*}
  \partial_t g = \dot g + \lambda_1'\partial_{\lambda_1}V(\lambda_1, \lambda_2) + \lambda_2'\big(\Lambda U_\uln{\lambda_2}^{a_2}
  +\partial_{\lambda_2} V(\lambda_1, \lambda_2)\big)- \big(a_2'+\frac{\nu}{\lambda_2}a_2\big) \partial_a U_{\lambda_2}^{a_2},
\end{equation*}
hence differentiating in time the first orthogonality relation $\la \frac{1}{\lambda_1}\cZ_\uln{\lambda_1}, g\ra = 0$ we obtain
  \begin{equation}
    \begin{aligned}
      0 = \dd t\la \cZ_\uln{\lambda_1}, g\ra &= -\frac{\lambda_1'}{\lambda_1^2}\la \Lambda_{-1}\cZ_\uln{\lambda_1}, g\ra + \frac{1}{\lambda_1}\la \cZ_\uln{\lambda_1}, \dot g\ra + \frac{\lambda_1'}{\lambda_1}\la\cZ_\uln{\lambda_1}, \partial_{\lambda_1}V(\lambda_1, \lambda_2)\ra \\
      &+ \frac{\lambda_2'}{\lambda_1}\la \cZ_\uln{\lambda_1}, \Lambda U_\uln{\lambda_2}^{a_2} +\partial_{\lambda_2} V(\lambda_1, \lambda_2)\ra - \frac{1}{\lambda_1}\Big(a_2' + \frac{\nu}{\lambda_2}a_2\Big)\la \cZ_\uln{\lambda_1}, \partial_a U_{\lambda_2}^{a_2}\ra,
    \end{aligned}
  \end{equation}
  which can also be written as
  \begin{equation}
    \label{eq:bulles-mod-orth-1}
    M_{11}\cdot \lambda_1' + \lambda M_{12}\cdot \lambda_2' + \lambda M_{13}\cdot \lambda_2\Big(a_2' + \frac{\nu}{\lambda_2}a_2\Big) = -\la \cZ_\uln{\lambda_1}, \dot g\ra,
  \end{equation}
  where for simplicity we write $M_{jk}$ instead of $M_{jk}(\bs g; \lambda_1, \lambda_2, a_2)$.
  Similarly, differentiating the second orthogonality relation $\la \frac{1}{\lambda_2}\cZ_\uln{\lambda_2}, g\ra = 0$ we obtain
  \begin{equation}
    \begin{aligned}
      0 = \dd t\la \cZ_\uln{\lambda_2}, g\ra &= -\frac{\lambda_2'}{\lambda_2^2}\la \Lambda_{-1}\cZ_\uln{\lambda_2}, g\ra + \frac{1}{\lambda_2}\la \cZ_\uln{\lambda_2}, \dot g\ra + \frac{\lambda_1'}{\lambda_2}\la\cZ_\uln{\lambda_2}, \partial_{\lambda_1}V(\lambda_1, \lambda_2)\ra \\
      &+ \frac{\lambda_2'}{\lambda_2}\la \cZ_\uln{\lambda_2}, \Lambda U_\uln{\lambda_2}^{a_2} +\partial_{\lambda_2} V(\lambda_1, \lambda_2)\ra - \frac{1}{\lambda_2}\Big(a_2' + \frac{\nu}{\lambda_2}a_2\Big)\la \cZ_\uln{\lambda_2}, \partial_a U_{\lambda_2}^{a_2}\ra,
    \end{aligned}
  \end{equation}
  which can also be written as
  \begin{equation}
    \label{eq:bulles-mod-orth-2}
    \frac{1}{\lambda}M_{21}\cdot \lambda_1' + M_{22}\cdot \lambda_2' + M_{23}\cdot \lambda_2\Big(a_2' + \frac{\nu}{\lambda_2}a_2\Big) = -\la \cZ_\uln{\lambda_2}, \dot g\ra.
  \end{equation}
  Finally, differentiating the third orthogonality relation $\la \alpha_{\lambda_2}^-, \bs g\ra = 0$ we obtain

  \begin{equation*}
    \begin{aligned}
      0 &= \dd t\la\alpha_{\lambda_2}^-, \bs g\ra = -\frac{\lambda_2'}{\lambda_2}\la\Lambda_{\cE^*}\alpha_{\lambda_2}^-, \bs g\ra +
      \la\alpha_{\lambda_2}^-, J\circ(\vD E(\bs U(\lambda_1, \lambda_2, a_2) + \bs g) - \vD E(\bs U_{\lambda_2}^{a_2}))\ra \\
      &+ \lambda_1'\la \alpha_{\lambda_2}^-, \partial_{\lambda_1} \bs V(\lambda_1, \lambda_2)\ra+\frac{\lambda_2'}{\lambda_2}\la\alpha_{\lambda_2}^-, \Lambda_\cE\bs U_{\lambda_2}^{a_2} + \lambda_2 \partial_{\lambda_2}\bs V(\lambda_1, \lambda_2)\ra - \Big(a_2' + \frac{\nu a_2}{\lambda_2}\Big)\big\la\alpha_{\lambda_2}^-, \partial_a \bs U_{\lambda_2}^{a_2}\big\ra,
  \end{aligned}
  \end{equation*}
  which can also be written as
  \begin{equation}
    \label{eq:bulles-mod-orth-3}
    \frac{1}{\lambda}M_{31}\cdot \lambda_1' + M_{32}\cdot \lambda_2' + M_{33}\cdot \lambda_2\Big(a_2' + \frac{\nu}{\lambda_2}a_2\Big)
    = -\lambda_2\la\alpha_{\lambda_2}^-, J\circ(\vD E(\bs U(\lambda_1, \lambda_2, a_2) + \bs g)-\vD E(\bs U_{\lambda_2}^{a_2})\ra.
  \end{equation}
  Equations \eqref{eq:bulles-mod-orth-1}, \eqref{eq:bulles-mod-orth-2} and \eqref{eq:bulles-mod-orth-3}
  form a linear system for $\lambda_1'$, $\lambda_2'$ and $\lambda_2\Big(a_2' + \frac{\nu}{\lambda_2}a_2\Big)$:

  \begin{equation*}
    \begin{pmatrix}
      M_{11} & \lambda M_{12} & \lambda M_{13} \\ \frac{1}{\lambda}M_{21} & M_{22} & M_{23} \\ \frac{1}{\lambda}M_{31} & M_{32} & M_{33}
    \end{pmatrix}
    \begin{pmatrix}
      \lambda_1' \\ \lambda_2' \\ a_2 + \frac{\lambda_2 a_2'}{\nu}
    \end{pmatrix} =
    \begin{pmatrix}
      -\la\cZ_\uln{\lambda_1}, \dot g\ra \\
      -\la\cZ_\uln{\lambda_2}, \dot g\ra \\
      -\lambda_2\la \alpha_{\lambda_2}^-, J\circ(\vD E(\bs U(\lambda_1, \lambda_2, a_2) + \bs g) - \vD E(\bs U_{\lambda_2}^{a_2}))\ra
    \end{pmatrix}.
  \end{equation*}

We will check that
\begin{equation}
  \label{eq:bulles-linear-right}
|\la \alpha_{\lambda_2}^-, J\circ(\vD E(\bs U(\lambda_1, \lambda_2, a_2) + \bs g) - \vD E(\bs U_{\lambda_2}^{a_2}))\ra|\lesssim \frac{1}{\lambda_2}\cN(\bs g, \lambda)(|a_2| + \cN(\bs g, \lambda)).
\end{equation}
By \eqref{eq:coer-alpha-miracle}, it suffices to show that
\begin{equation}
  \label{eq:bulles-infl-bulle1}
|\la \alpha_{\lambda_2}^-, J\circ(\vD E(\bs U(\lambda_1, \lambda_2, a_2) + \bs g) - \vD E(\bs U^{a_2}_{\lambda_2} + \bs g))\ra| \lesssim \frac{1}{\lambda_2}\cN(\bs g, \lambda)^2.
\end{equation}
Without loss of generality we can assume that $\lambda_2 = 1$ and $\lambda_1 = \lambda$,
hence \eqref{eq:bulles-infl-bulle1} is equivalent to
\begin{equation}
  \label{eq:bulles-destab-01}
|\la \cY, -\Delta V(\lambda, 1) + f(-V(\lambda, 1) + U^{a_2} + g) - f(U^{a_2} + g)\ra| \lesssim \cN(\bs g, \lambda)^2.
\end{equation}
We have
$$|\la \cY, \Delta V(\lambda, 1)\ra| = |\la \Delta \cY, V(\lambda, 1)\ra| \lesssim \lambda^\frac{N-2}{2}$$
because of \eqref{eq:db-lV-L1}.
For the other term we use the bound
$$
|f(-V(\lambda, 1) + U^{a_2} + g) - f(U^{a_2} + g)| \lesssim (f'(U^{a_2}) + f'(g))V(\lambda, 1) + f(V(\lambda, 1)).
$$
From \eqref{eq:db-V-L1} we obtain
$|\la \cY, f'(U^{a_2})V(\lambda, 1)\ra| \lesssim \|V(\lambda, 1)\|_{L^1} \lesssim \cN(\bs g, \lambda)^2$.
Using H\"older we compute
$$
|\la \cY, f'(g)\cdot V(\lambda, 1)\ra| \lesssim \|f'(g)\|_{L^\frac N2}\cdot \|V(\lambda, 1)\|_{L^{\frac{N}{N-2}}}\lesssim
\|\bs g\|_\cE^\frac{4}{N-2}\cdot \lambda^\frac{N-2}{2}|\log\lambda| \lesssim \cN(\bs g, \lambda)^2.
$$
Finally, $|\la \cY, f(V(\lambda, 1))\ra| \lesssim \|f(W_\lambda)\|_{L^1} \lesssim \lambda^\frac{N-2}{2}$.
This finishes the proof of \eqref{eq:bulles-destab-01}, hence we have shown \eqref{eq:bulles-linear-right}.

Consider the inverse matrix
\begin{equation*}
  \begin{pmatrix}
    P_{11} & P_{12} & P_{13} \\ P_{21} & P_{22} & P_{23} \\ P_{31} & P_{32} & P_{33}
  \end{pmatrix} :=
  \begin{pmatrix}
    M_{11} & \lambda M_{12} & \lambda M_{13} \\ \frac{1}{\lambda}M_{21} & M_{22} & M_{23} \\ \frac{1}{\lambda}M_{31} & M_{32} & M_{33}
  \end{pmatrix}^{-1}.
\end{equation*}
From \eqref{eq:bulles-triangular} we obtain
  \begin{equation}
    \label{eq:bulles-triangular-2}
  \begin{aligned}
    &|P_{11}| \lesssim &1,\quad &|P_{12}| \lesssim &1,\quad &|P_{13}| \lesssim &1, \\
    &|P_{21}| \lesssim &1,\quad &|P_{22}| \lesssim &1,\quad &|P_{23}| \lesssim &1, \\
    &|P_{31}| \lesssim &\cN(\bs g, \lambda) + |a_2|,\quad &|P_{32}| \lesssim &\cN(\bs g, \lambda) + |a_2|,\quad &|P_{33}| \lesssim &1,
  \end{aligned}
\end{equation}
hence \eqref{eq:bulles-linear-right} yields \eqref{eq:bulles-mod} and \eqref{eq:bulles-mod-a}.

\end{proof}

We finish this subsection by analyzing the stability and instability components at both scales $\lambda_1(t)$ and $\lambda_2(t)$.
At the scale $\lambda_2(t)$ we use the refined component $\beta_{\lambda_2}^{a_2}$ introduced in Section~\ref{sec:coer}, see \eqref{eq:beta}.
\begin{proposition}
  \label{prop:bulles-dtb}
  The functions
  $$
  a_1^-(t) := \la \alpha_{\lambda_1(t)}^-, \bs g(t)\ra,\qquad a_1^+(t) := \la \alpha_{\lambda_1(t)}^+, \bs g(t)\ra, \qquad b_2(t):= \la \beta_{\lambda_2(t)}^{a_2(t)}, \bs g(t)\ra
  $$
  satisfy
  \begin{align}
    \big|\dd t a_1^-(t) + \frac{\nu}{\lambda_1(t)}a_1^-(t)\big| &\lesssim \frac{1}{\lambda_1(t)}\cN(\bs g(t), \lambda(t))^2, \label{eq:bulles-dtAp} \\
    \big|\dd t a_1^+(t) - \frac{\nu}{\lambda_1(t)}a_1^+(t)\big| &\lesssim \frac{1}{\lambda_1(t)}\cN(\bs g(t), \lambda(t))^2, \label{eq:bulles-dtAm} \\
    \big|\dd t b_2(t) - \frac{\nu}{\lambda_2(t)}b_2(t)\big| &\lesssim \frac{1}{\lambda_2(t)}\cN(\bs g(t), \lambda(t))^2 \label{eq:bulles-dtb},
  \end{align}
  with constants eventually depending on $R$.
\end{proposition}
\begin{proof}~
  \paragraph{\textbf{Step 1.}}
  Directly from the definition of $a_1^-(t)$ we obtain
  \begin{equation*}
    \dd t a_1^-(t) = -\frac{\lambda_1'(t)}{\lambda_1(t)}\la \Lambda_{\cE^*}\alpha_{\lambda_1(t)}^-, \bs g(t)\ra + \la \alpha_{\lambda_1(t)}^-, \partial_t \bs g(t)\ra.
  \end{equation*}
  The first term is negligible due to \eqref{eq:bulles-mod}. We compute the second term using \eqref{eq:bulles-mod-g}.
  We begin by treating the terms in the second line of \eqref{eq:bulles-mod-g}. Since $|\lambda_1'| + |\lambda_2'| \lesssim 1$ and $\big|a_2'+\frac{\nu}{\lambda_2}a_2\big| \lesssim \frac{1}{\lambda_2}$
  (of course Lemma~\ref{lem:bulles-mod} provides better estimates, but we do not need it here), it suffices to check that
  \begin{equation}
    \label{eq:bulles-destab-1}
    \begin{aligned}
    &|\la\alpha_{\lambda_1}^-, \partial_{\lambda_1}\bs V(\lambda_1, \lambda_2)\ra| + |\la\alpha_{\lambda_1}^-, \partial_{\lambda_2}\bs V(\lambda_1, \lambda_2)\ra| \\
    + &|\la \alpha_{\lambda_1}^-, \frac{1}{\lambda_2}\Lambda_\cE \bs U_{\lambda_2}^{a_2}\ra| + |\la \alpha_{\lambda_1}^-, \frac{1}{\lambda_2}\partial_{a_2} \bs U_{\lambda_2}^{a_2}\ra| \lesssim \frac{1}{\lambda_1}\cdot\big(\frac{\lambda_1}{\lambda_2}\big)^\frac{N-2}{2}.
  \end{aligned}
  \end{equation}
  The estimate is invariant by rescaling both $\lambda_1$ and $\lambda_2$, hence we can assume that $\lambda_2 = 1$ and $\lambda_1 = \lambda$.
  For the first term we use \eqref{eq:db-lV-Linf} and rapid decay of $\cY$. Estimating the other terms is straightforward.

  Now consider the first line of \eqref{eq:bulles-mod-g}. It follows from \eqref{eq:eigencovectl} that it suffices to show that
  \begin{equation}
    \big|\big\la \alpha_{\lambda_1}^-, J\circ\big(\vD E(\bs U(\lambda_1, \lambda_2, a_2) + \bs g) - \vD E(\bs U_{\lambda_2}^{a_2}) - \vD^2 E(\bs W_{\lambda_1})\bs g\big)\big\ra\big| \lesssim \frac{1}{\lambda_1}\cN(\bs g, \lambda)^2,
  \end{equation}
  which is equivalent to
  \begin{equation}
    \label{eq:bulles-destab-2}
    |\la \cY_{\lambda_1}, f(U(\lambda_1, \lambda_2, a_2) + g) - f(U_{\lambda_2}^{a_2}) - \Delta V(\lambda_1, \lambda_2) - f'(W_{\lambda_1})g\ra| \lesssim \cN(\bs g, \lambda)^2.
  \end{equation}
  We can assume that $\lambda_2 = 1$ and $\lambda_1 = \lambda$. By the triangle inequality, it suffices to check that
  \begin{align}
    \label{eq:bulles-destab-3}
    |\la \cY_{\lambda}, \Delta V(\lambda, 1) + f(V(\lambda, 1))\ra| &\lesssim \cN(\bs g, \lambda)^2, \\
    \label{eq:bulles-destab-4}
    |\la \cY_{\lambda}, f(U(\lambda, 1, a_2) + g) - f(U(\lambda, 1, a_2)) - f'(U(\lambda, 1, a_2))g\ra| &\lesssim \cN(\bs g, \lambda)^2, \\
    \label{eq:bulles-destab-5}
    |\la \cY_{\lambda}, f(U(\lambda, 1, a_2)) - f(U^{a_2}) + f(V(\lambda, 1))\ra| &\lesssim \cN(\bs g, \lambda)^2, \\
    \label{eq:bulles-destab-6}
    |\la \cY_{\lambda}, \big(f'(U(\lambda, 1, a_2))\big) - f'(W_\lambda)\big)g\ra| &\lesssim \cN(\bs g, \lambda)^2.
  \end{align}

  Notice that $|f(W_\lambda) - f(V(\lambda, 1))| \lesssim f'(W_\lambda)|\cdot |W_\lambda - V(\lambda)| \lesssim f'(W_\lambda)$,
  where the last inequality follows from \eqref{eq:db-V-Linf}. Together with the fact that $\Delta(W_\lambda) + f(W_\lambda) = 0$ this implies
  \begin{equation*}
    \begin{aligned}
      \big|\la \cY_\lambda, \big(\Delta V(\lambda, 1) + f(V(\lambda, 1))\big)\big\ra\big| &\lesssim \big|\big\la \cY_\lambda, \Delta\big(W_\lambda - V(\lambda, 1)\big)\big\ra\big| + \big|\big\la \cY_\lambda, f(W_\lambda) - f(V(\lambda, 1))\big\ra\big| \\
      &\lesssim\|\Delta Y_\lambda\|_{L^1} + \|f'(W_\lambda)\cY_\lambda\|_{L^1} \lesssim \lambda^\frac{N-2}{2},
    \end{aligned}
  \end{equation*}
  which proves \eqref{eq:bulles-destab-3}.

  To fix ideas, notice that while proving the remaining inequalities we can restrict our attention to the region $|x| \leq c\sqrt\lambda$ where $c > 0$ is a small constant
  (the region $|x| \geq c\sqrt\lambda$ is negligible thanks to the rapid decay of $\cY$).
  In this region we have $W_\lambda \geq V(\lambda, 1) \gtrsim 1$ and $|U(\lambda, 1, a_2) + W_\lambda| \leq \frac 12 W_\lambda$ pointwise.

  Inequality \eqref{eq:bulles-destab-5} follows immediately from
  \begin{equation}
    \label{eq:pointwise-ustar-V}
    |f(U(\lambda, 1, a_2)) - f(U^{a_2}) + f(V(\lambda, 1))| = |f(U^{a_2} - V(\lambda, 1)) - f(U^{a_2}) + f(V(\lambda, 1))| \lesssim f'(W_\lambda).
  \end{equation}

  We have the bound
$$
|f'(U(\lambda, 1, a_2)) - f'(W_\lambda)| \lesssim (|f''(W_\lambda)| + |f''(U(\lambda, 1, a_2))|)\cdot|U(\lambda, 1, a_2) + W_\lambda| \lesssim |f''(W_\lambda)|
$$
(even in the case $N \geq 6$ when $f''$ is a negative power).
Using H\"older and the fact that $\|\cY_\lambda\cdot f''(W_\lambda)\|_{L^\frac{2N}{N+2}} \lesssim \lambda^\frac{N-2}{2}$, this implies \eqref{eq:bulles-destab-6}

  For \eqref{eq:bulles-destab-4}, we consider separately the cases $N \in \{3, 4, 5\}$ and $N \geq 6$.
  In the first case, \eqref{eq:bulles-destab-4} follows from the pointwise bound
  \begin{equation}
    \label{eq:pointwise-ustar-V-g}
    |f(U(\lambda, 1, a_2) + g) - f(U(\lambda, 1, a_2)) - f'(U(\lambda, 1, a_2))g| \lesssim |f''(U(\lambda, 1, a_2))|\cdot|g|^2 + f(|g|).
  \end{equation}
In the case $N \geq 6$ we still have
  \begin{equation}
    |f(U(\lambda, 1, a_2) + g) - f(U(\lambda, 1, a_2)) - f'(U(\lambda, 1, a_2))g| \lesssim |f''(U(\lambda, 1, a_2))|\cdot|g|^2,
  \end{equation}
  even if $f''$ is a negative power. This yields \eqref{eq:bulles-destab-4}.

  This finishes the proof of \eqref{eq:bulles-dtAp} and the proof of \eqref{eq:bulles-dtAm} is almost the same.
  \paragraph{\textbf{Step 2.}}
  The proof of \eqref{eq:bulles-dtb} is close to the proof of Proposition~\ref{prop:coer-dtb}, but there will be more error terms to estimate. First we need to show that
  \begin{equation}
    |\la \beta_{\lambda_2}^{a_2} - \alpha_{\lambda_2}^+, \partial_t \bs g\ra| \lesssim \frac{1}{\lambda_2}|a_2|\cdot \cN(\bs g, \lambda). \label{eq:beta-alpha-3}
  \end{equation}
  Since $\|\beta_1^{a_2} - \alpha^+\|_{L^\infty \times L^\infty} \lesssim |a_2|$, the proof of \eqref{eq:bulles-infl-bulle1} gives
$$
|\la \beta_{\lambda_2}^{a_2} - \alpha_{\lambda_2}^+, J\circ(\vD E(\bs U(\lambda_1, \lambda_2, a_2) + \bs g) - \vD E(\bs U^{a_2}_{\lambda_2} + \bs g))\ra| \lesssim \frac{1}{\lambda_2}|a_2|\cdot \cN(\bs g, \lambda)^2 \ll \frac{1}{\lambda_2}|a_2|\cdot \cN(\bs g, \lambda).
$$
 Using \eqref{eq:coer-beta-alpha-4}, we obtain
 $$
|\la \beta_{\lambda_2}^{a_2} - \alpha_{\lambda_2}^+, J\circ(\vD E(\bs U(\lambda_1, \lambda_2, a_2) + \bs g) - \vD E(\bs U^{a_2}_{\lambda_2}))\ra| \lesssim \frac{1}{\lambda_2}|a_2|\cdot \cN(\bs g, \lambda).
 $$
Similarly one obtains
$$
|\la \beta_{\lambda_2}^{a_2} - \alpha_{\lambda_2}^+, \partial_{\lambda_1}\bs V(\lambda_1, \lambda_2)\ra| + |\la \beta_{\lambda_2}^{a_2} - \alpha_{\lambda_2}^+, \partial_{\lambda_2}\bs V(\lambda_1, \lambda_2)\ra| \ll \frac{1}{\lambda_2} |a_2|,
$$
hence \eqref{eq:beta-alpha-3} follows from \eqref{eq:coer-beta-alpha-5}, \eqref{eq:bulles-mod} and \eqref{eq:bulles-mod-a}.

  \paragraph{\textbf{Step 3.}}
  Suppose that
\begin{equation}
  \label{eq:bulles-a-leq-g}
|a_2(t)| \leq \cN(\bs g(t), \lambda(t)).
\end{equation}
We have
\begin{equation}
  \label{eq:bulles-dtb-1}
  \dd t b_2(t) = \la \beta_{\lambda_2(t)}^{a_2(t)}, \partial_t \bs g(t)\ra + \lambda_2'(t)\la \partial_\lambda \beta_{\lambda_2(t)}^{a_2(t)}, \bs g(t)\ra + a_2'(t)\la \partial_a \beta_{\lambda_2(t)}^{a_2(t)}, \bs g\ra.
\end{equation}
From Lemma~\ref{lem:bulles-mod} we know that $|\lambda_2'| \lesssim \cN(\bs g, \lambda)$ and $|a_2'| \lesssim \frac{1}{\lambda_2}\cN(\bs g, \lambda)$. Hence from \eqref{eq:deriv-beta}
it follows that the last two terms of \eqref{eq:bulles-dtb} are negligible.

Using \eqref{eq:bulles-a-leq-g}, \eqref{eq:beta-alpha-1} and \eqref{eq:beta-alpha-3} we see that it is sufficient to show that
\begin{equation}
  \label{eq:bulles-dtalpha}
  \big|\la \alpha_{\lambda_2}^+, \partial_t \bs g\ra - \frac{\nu}{\lambda_2}\la \alpha_{\lambda_2}^+, \bs g\ra\big| = |\la \alpha_{\lambda_2}^+, \partial_t \bs g - J\circ\vD^2 E(\bs W_{\lambda_2})\bs g\ra| \lesssim \frac{1}{\lambda_2}\cN(\bs g, \lambda)^2.
\end{equation}
We develop $\partial_t \bs g$ using \eqref{eq:bulles-mod-g}. Consider first the terms in the second line of \eqref{eq:bulles-mod-g}.
From \eqref{eq:db-lV-L1} and \eqref{eq:bulles-mod} we have
$$|\la \alpha_{\lambda_2}^+, \lambda_1'\partial_{\lambda_1}\bs V(\lambda_1, \lambda_2)\ra| \lesssim \frac{1}{\lambda_2}\cN(\bs g, \lambda)^2.$$
Since $|\partial_{\lambda_2}V(\lambda_1, \lambda_2)| \lesssim \lambda_2^{-\frac N2}$, see \eqref{eq:bulles-Vl2}, using \eqref{eq:bulles-mod} we get
$$|\la \alpha_{\lambda_2}^+, \lambda_2'\partial_{\lambda_1}\bs V(\lambda_1, \lambda_2)\ra| \lesssim \frac{1}{\lambda_2}\cN(\bs g, \lambda)^2.$$
The other two terms have already appeared in the proof of Proposition~\ref{prop:coer-dtb}, see \eqref{eq:coer-dtalpha}.

Consider now the first line of \eqref{eq:bulles-mod-g}. From Lemma~\ref{lem:weak-linea} we deduce that
$$
|\la \alpha_{\lambda_2}^+, J\circ(\vD E(\bs U^{a_2}_{\lambda_2} + \bs g) - \vD E(\bs U^{a_2}_{\lambda_2}) - \vD^2 E(\bs W_{\lambda_2})\bs g)\ra| \lesssim \frac{1}{\lambda_2}\cN(\bs g, \lambda)^2,
$$
hence it suffices to check that
$$
|\la \alpha_{\lambda_2}^+, J\circ(\vD E(\bs U(\lambda_1, \lambda_2, a_2) + \bs g) - \vD E(\bs U^{a_2}_{\lambda_2} + \bs g))\ra| \lesssim \frac{1}{\lambda_2}\cN(\bs g, \lambda)^2,
$$
whose proof is the same as the proof of \eqref{eq:bulles-infl-bulle1}.
\paragraph{\textbf{Step 4.}}
Now we consider the case
\begin{equation}
  \label{eq:bulles-g-leq-a}
  \cN(\bs g(t), \lambda(t)) \leq |a_2(t)|,
\end{equation}
in particular $a_2 \neq 0$.

Recall that (see Proposition~\ref{prop:Ua})
\begin{equation}
  \label{eq:bulles-beta-form-2}
  \beta_{\lambda_2}^{a_2} = -\frac{1}{2a_2}\vD E(\bs U_{\lambda_2}^{a_2}) \quad \Rightarrow \quad b_2(t) = -\frac{1}{2a_2(t)}\cdot \la \vD E(\bs U_{\lambda_2(t)}^{a_2(t)}), \bs g(t)\ra.
\end{equation}
From \eqref{eq:bulles-mod-a} and \eqref{eq:bulles-g-leq-a} we obtain $\big|\frac{a_2'(t)}{a_2(t)} + \frac{\nu}{\lambda_2(t)}\big| \lesssim \frac{1}{\lambda_2(t)}\cN(g(t), \lambda(t))$, hence
\begin{equation*}
  \begin{aligned}
  \dd t b_2(t) &= -\frac{a_2'(t)}{a_2(t)}b_2(t) -\frac{1}{2a_2(t)}\dd t\big\la\vD E(\bs U_{\lambda(t)}^{a_2(t)}),\bs g(t)\big\ra \\ 
  &= \frac{\nu}{\lambda_2(t)}b_2(t) -\frac{1}{2a_2(t)} \dd t\big\la \vD E(\bs U_{\lambda_2(t)}^{a_2(t)}),\bs g(t)\big\ra + \frac{1}{\lambda_2(t)}O(\cN(\bs g(t), \lambda(t))^2).
\end{aligned}
  \end{equation*}
  We compute the second term using \eqref{eq:bulles-mod-g} and \eqref{eq:Ua-crit}:
  \begin{equation}
    \label{eq:bulles-dtb-fin0}
    \begin{aligned}
      &\dd t\la \vD E(\bs U_{\lambda_2}^{a_2}), \bs g\ra = \la \vD^2 E(\bs U_{\lambda_2}^{a_2})\partial_t \bs U_{\lambda_2}^{a_2}, \bs g\ra + \la \vD E(\bs U_\lambda^a), \\
      &\quad J\circ(\vD E(\bs U(\lambda_1, \lambda_2, a_2) + \bs g) - \vD E(\bs U_{\lambda_2}^{a_2})) + \lambda_1'\partial_{\lambda_1}\bs V(\lambda_1, \lambda_2) + \lambda_2'\partial_{\lambda_2} \bs V(\lambda_1, \lambda_2)\ra.
  \end{aligned}
  \end{equation}
  We have to prove that $\big|\dd t\la \vD E(\bs U_{\lambda_2}^{a_2}), \bs g\ra\big| \lesssim \frac{a_2}{\lambda_2}\cN(\bs g, \lambda)^2$. Until the end of this proof ``negligible'' means $\lesssim \frac{a_2}{\lambda_2}\cN(\bs g, \lambda)^2$.

  From \eqref{eq:db-lV-L1} and \eqref{eq:bulles-Vl2} it follows that
  $$
  |\la\vD(\bs U_{\lambda_2}^{a_2}), \partial_{\lambda_1}\bs V(\lambda_1, \lambda_2)\ra| \lesssim \frac{1}{\lambda_2}\lambda^\frac{N-2}{2}, \\
  |\la\vD(\bs U_{\lambda_2}^{a_2}), \partial_{\lambda_2}\bs V(\lambda_1, \lambda_2)\ra| \lesssim \frac{1}{\lambda_2}\lambda^\frac{N}{2}.
$$
By \eqref{eq:bulles-mod} and \eqref{eq:bulles-g-leq-a}, the contribution of the last two terms in \eqref{eq:bulles-dtb-fin0} is negligible.

Next, we will show that
$$
|\la \vD(\bs U_{\lambda_2}^{a_2}), J\circ(\vD E(\bs U(\lambda_1, \lambda_2, a_2) + \bs g) - \vD E(\bs U_{\lambda_2}^{a_2} + \bs g))\ra| \lesssim \frac{a_2}{\lambda_2}\cN(\bs g, \lambda)^2.
$$
We can assume that $\lambda_2 = 1$ and $\lambda_1 = \lambda$, hence we have to prove that
\begin{equation}
  \label{eq:bulles-dtb-fin1}
|\la \dot U^{a_2}, f(U(\lambda, 1, a_2)+g) - f(U_{\lambda_2}^{a_1} + g)\ra| \lesssim a_2 \cN(\bs g, \lambda)^2.
\end{equation}
In the region $|x| > R\sqrt{\lambda}$ the integrand equals $0$. In the region $|x| \leq R\sqrt{\lambda}$ we have a pointwise bound
$$
|f(U(\lambda, 1, a_2)+g) - f(U^{a_2} + g)| \lesssim f'(U^{a_2} + g)W_\lambda + f(W_\lambda) \lesssim (f'(U^{a_2}) + f'(g))W_\lambda + f(W_\lambda).
$$
Recall that $\|\dot U^{a_2}\|_{L^\infty} \lesssim |a_2|$ and $\|U^{a_2}\|_{L^\infty} \lesssim 1$. Thus
\begin{align}
  |\la |\dot U^{a_2}|, f'(U^{a_2})W_\lambda\ra| &\lesssim |a_2|\cdot \|W_\lambda\|_{L^1(|x| \leq R\sqrt\lambda)} \sim |a_2|\lambda^\frac{N-2}{2}, \\
  |\la |\dot U^{a_2}|, f'(g)W_\lambda\ra| &\lesssim |a_2|\cdot \|f'(g)\|_{L^\frac N2}\cdot \|W_\lambda\|_{L^\frac{N}{N-2}(|x| \leq R\sqrt\lambda)} \\
  &\lesssim |a_2|\cdot \|g\|_{\dot H^1}^\frac{4}{N-2}\cdot \lambda^\frac{N-2}{2}|\log \lambda| \lesssim |a_2|\cN(\bs g, \lambda)^2, \\
  |\la |\dot U^{a_2}|, f(W_\lambda)\ra| &\lesssim |a_2|\cdot \|f(W_\lambda)\|_{L^1} \sim |a_2|\lambda^\frac{N-2}{2}.
\end{align}
This proves \eqref{eq:bulles-dtb-fin1}.

In order to finish the proof, it suffices to check that
$$
|\la \vD^2 E(\bs U_{\lambda_2}^{a_2})\partial_t \bs U_{\lambda_2}^{a_2}, \bs g\ra + \la \vD E(\bs U_{\lambda_2}^{a_2}), J\circ(\vD E(\bs U_{\lambda_2}^{a_2} + \bs g) - \vD E(\bs U_{\lambda_2}^{a_2}))\ra| \lesssim \frac{a_2}{\lambda_2}\cN(\bs g, \lambda)^2,
$$
which is achieved exactly as in the last part of the proof of Proposition~\ref{prop:coer-dtb}.
\end{proof}

\subsection{Coercivity near the sum of two bubbles}
\label{ssec:bulles-coer}
We have the following analogue of Lemma~\ref{lem:coer}:
\begin{lemma}
  \label{lem:bulles-coer}
  There exists constants $\lambda_0, \eta > 0$ such that if $\lambda = \frac{\lambda_1}{\lambda_2} < \lambda_0$ and $\|\bs U - (\bs W_{\lambda_2} - \bs W_{\lambda_1})\|_\cE < \eta$,
  then for all $\bs g \in \cE$ such that $\la \cZ_\uln{\lambda_1}, g\ra =\la \cZ_\uln{\lambda_2}, g\ra = 0$ there holds
    \begin{equation*}
      \frac 12 \la \vD^2 E(\bs U)\bs g, \bs g\ra +2\big(\la \alpha^-_{\lambda_1}, \bs g\ra^2 + \la \alpha^+_{\lambda_1}, \bs g\ra^2 + \la \alpha^-_{\lambda_2}, \bs g\ra^2 + \la \alpha^+_{\lambda_2}, \bs g\ra^2\big) \gtrsim \|\bs g\|_\cE^2.
    \end{equation*}
\end{lemma}
\begin{proof}~
  \paragraph{\textbf{Step 1.}}
  Without loss of generality we can assume that $\lambda_2 = 1$ and $\lambda_1 = \lambda$.
  Consider the operator $H_\lambda$ defined by the following formula:
  \begin{equation*}
    H_\lambda := \begin{pmatrix} -\Delta - f'(W_{\lambda}) - f'(W) & 0 \\ 0 & \Id\end{pmatrix}.
  \end{equation*}
  We will show that for any $c > 0$ there holds
  \begin{equation}
    \label{eq:bulles-coer-approx}
    |\la \vD^2 E(\bs U)\bs g, \bs g\ra - \la H_{\lambda}\bs g, \bs g\ra| \leq c\|\bs g\|_\cE^2,\qquad \forall \bs g\in\cE,
  \end{equation}
  provided that $\eta$ and $\lambda_0$ are small enough. By H\"older and Sobolev, it suffices (eventually changing $c$) to check that
  \begin{equation}
    \label{eq:bulles-coer-approx-2}
    \|f'(U) - f'(W_{\lambda}) - f'(W)\|_{L^\frac N2} \leq c.
  \end{equation}
  Since (by pointwise estimates)
  \begin{equation*}
    \|f'(U) - f'(W - W_{\lambda})\|_{L^\frac N2} \lesssim \max(\eta, f'(\eta)),
  \end{equation*}
  this will in turn follow from
  \begin{equation}
    \label{eq:bulles-coer-approx-3}
    \|f'(W - W_{\lambda}) - f'(W_{\lambda}) - f'(W)\|_{L^\frac N2} \leq c.
  \end{equation}
  We consider separately the regions $|x| \leq \sqrt\lambda$ and $|x| \geq \sqrt\lambda$.
  In both cases we will use the fact that
  \begin{equation}
    \label{eq:bulles-coer-pointwise}
    \begin{aligned}
      |l| \lesssim |k| \quad&\Rightarrow\quad \big|f'(k+l) - f'(k) - f'(l)\big| \lesssim f'(l), \qquad\text{for }N\geq 6, \\
      |l| \lesssim |k| \quad&\Rightarrow\quad \big|f'(k+l) - f'(k) - f'(l)\big| \lesssim |f''(k)|\cdot |l|, \qquad\text{for }N \in \{3, 4, 5\}.
    \end{aligned}
  \end{equation}
  In the region $|x| \leq \sqrt\lambda$ we have $W \lesssim W_\lambda$, hence by \eqref{eq:bulles-coer-pointwise}
  $$
  \big| f'(W-W_\lambda) - f'(W_\lambda) - f'(W)\big| \lesssim 1,
  $$
  and $\|1\|_{L^\frac N2(|x| \leq \sqrt\lambda)} \sim \lambda$.

  In the region $|x| \geq \sqrt\lambda$ we have $W_\lambda \lesssim W$. If $N \geq 6$, then
  $$
  \big| f'(W - W_\lambda) - f'(W_\lambda) - f'(W)\big| \lesssim f'(W_\lambda).
  $$
  It is easy to check that $\|f'(W_\lambda)\|_{L^\frac N2(|x| \geq \sqrt\lambda)} \sim \lambda$.
  If $N \in \{3, 4, 5\}$, we obtain
  $$
  \big| f'(W - W_\lambda) - f'(W_\lambda) - f'(W)\big| \lesssim |f''(W)|\cdot |W_\lambda|,
  $$
  hence
  $$
  \|f'(W - W_\lambda) - f'(W_\lambda) - f'(W)\|_{L^\frac N2(|x| \geq \sqrt\lambda)} \lesssim \|f''(W)\|_{L^\frac{2N}{6-N}}\cdot \|W_\lambda\|_{L^\frac{2N}{N-2}(|x| \geq \sqrt\lambda)} \sim \lambda^\frac{N-2}{4}.
  $$
  This finishes the proof of \eqref{eq:bulles-coer-approx-3}.
  \paragraph{\textbf{Step 2.}}
  In view of \eqref{eq:bulles-coer-approx}, it suffices to prove that if $\lambda < \lambda_0$ and $\la \cZ, g\ra = \la \cZ_\uln\lambda, g\ra = 0$, then
    \begin{equation*}
      \frac 12 \la H_\lambda\bs g, \bs g\ra +2\big(\la \alpha^-_{\lambda_1}, \bs g\ra^2 + \la \alpha^+_{\lambda_1}, \bs g\ra^2 + \la \alpha^-_{\lambda_2}, \bs g\ra^2 + \la \alpha^+_{\lambda_2}, \bs g\ra^2\big) \gtrsim \|\bs g\|_\cE^2.
    \end{equation*}
  Let $a_1^- :=\la \alpha_\lambda^-, \bs g\ra$, $a_1^+ := \la \alpha_\lambda^+, \bs g\ra$, $a_2^- := \la\alpha^-, \bs g\ra$, $a_2^+ := \la\alpha^+, \bs g\ra$
  and decompose
  $$\bs g = a_1^-\cY_\lambda^- + a_1^+\cY_\lambda^+ + a_2^-\ym + a_2^+\yp + a_2^-\cY_\lambda^- +\bs k.$$
  Using the fact that
  $$
  \begin{aligned}
  |\la\alpha^\pm, \cY^\pm_\lambda\ra| + |\la \alpha_\lambda^\pm, \cY^\pm\ra| + |\la \frac{1}{\lambda}\cZ_\uln\lambda, \cY\ra| + |\la \cZ, \cY_\lambda\ra| &\lesssim \lambda^\frac{N-2}{2}, \\
  |a_1^-| + |a_1^+| + |a_2^-| + |a_2^+| &\lesssim \|\bs g\|_\cE, \\
  \la \alpha^-, \cY^+\ra = \la \alpha^+, \cY^-\ra = \la \cZ, \cY \ra &= 0
  \end{aligned}
  $$ we obtain
  \begin{equation}
    \label{eq:bulles-coer-eigendir}
    \la\alpha^-, \bs k\ra^2 + \la\alpha^+, \bs k\ra^2 + \la \alpha_\lambda^-, \bs k\ra^2 + \la \alpha_\lambda^+, \bs k\ra^2 + \la \cZ, k\ra^2 + \la \frac{1}{\lambda}\cZ_\uln\lambda, k\ra^2 \lesssim \lambda^{N-2}\|\bs g\|_\cE^2.
  \end{equation}

  Since $H_\lambda$ is self-adjoint, we can write
\begin{equation}
  \label{eq:bulles-coer-expansion}
  \begin{aligned}
  \frac 12 \la H_\lambda \bs g, \bs g\ra &= \frac 12 \la H_\lambda\bs k, \bs k\ra + \la H_\lambda(a_2^-\cY^- + a_2^+\cY^+), \bs k\ra + \la H_\lambda(a_1^-\cY_\lambda^- + a_1^+\cY_\lambda^+), \bs k\ra \\
  &+ \frac 12 \la H_\lambda(a_2^-\cY^- + a_2^+\cY^+), a_2^-\cY^- + a_2^+\cY^+\ra \\
  &+ \frac 12 \la H_\lambda(a_1^-\cY_\lambda^- + a_1^+\cY_\lambda^+), a_1^-\cY_\lambda^- + a_1^+\cY_\lambda^+\ra \\
  &+ \la H_\lambda(a_2^-\cY^- + a_2^+\cY^+), a_1^-\cY_\lambda^- + a_1^+\cY_\lambda^+\ra.
\end{aligned}
\end{equation}
It is easy to see that $\|f'(W)\cY_\lambda\|_{L^\frac{2N}{N+2}} \to 0$ and $\|f'(W_\lambda)\cY\|_{L^\frac{2N}{N+2}} \to 0$ as $\lambda \to 0$. This and \eqref{eq:al}, \eqref{eq:eigenvectl} imply
$$
\|H_\lambda \cY^- + 2\alpha^+\|_{\cE^*} + \|H_\lambda \cY^+ + 2\alpha^-\|_{\cE^*} + \|H_\lambda \cY_\lambda^- + 2\alpha_\lambda^+\|_{\cE^*} + \|H_\lambda \cY_\lambda^+ + 2\alpha_\lambda^-\|_{\cE^*} \sto{\lambda \to 0} 0.
$$
Plugging this into \eqref{eq:bulles-coer-expansion} and using \eqref{eq:bulles-coer-eigendir} we obtain
  \begin{equation}
    \label{eq:bulles-coer-approx-6}
    \frac 12\la H_\lambda \bs g, \bs g\ra \geq -2a_2^-a_2^+ - 2a_1^-a_1^+ + \frac 12\la H_\lambda\bs k, \bs k\ra -\wt c \|\bs g\|_\cE^2,
  \end{equation}
  where $\wt c \to 0$ as $\lambda \to 0$.
 
  Applying \eqref{eq:coer-lin-coer-2} with $r_1 = \lambda^{-\frac 12}$, rescaling and using \eqref{eq:bulles-coer-eigendir} we get, for $\lambda$ small enough,
  \begin{equation}
    \label{eq:bulles-coer-approx-4}
      (1-2c)\int_{|x|\leq \sqrt\lambda}|\grad k|^2 \ud x + c\int_{|x|\geq \sqrt\lambda}|\grad k|^2\ud x - \int_{\bR^N}f'(W_\lambda)|k|^2\ud x \geq -\wt c \|\bs g\|_\cE^2.
  \end{equation}
  From \eqref{eq:coer-lin-coer-3} with $r_2 = \sqrt\lambda$ we have
  \begin{equation}
    \label{eq:bulles-coer-approx-5}
      (1-2c)\int_{|x|\geq \sqrt\lambda}|\grad k|^2 \ud x + c\int_{|x|\leq \sqrt\lambda}|\grad k|^2\ud x - \int_{\bR^N}f'(W)|k|^2\ud x \geq -\wt c \|\bs g\|_\cE^2.
  \end{equation}
  Taking the sum of \eqref{eq:bulles-coer-approx-4} and \eqref{eq:bulles-coer-approx-5}, and using \eqref{eq:bulles-coer-approx-6} we obtain
  \begin{equation*}
    \frac 12\la H_\lambda \bs g, \bs g\ra \geq -2a_2^-a_2^+ - 2a_1^-a_1^+ + c\|\bs k\|_\cE^2 - 2\wt c\|\bs g\|_\cE^2.
  \end{equation*}
  The conclusion follows if we take $\wt c$ small enough.
\end{proof}

Recall that $R> 0$ is the constant used in the definition of the localized bubble $\bs V(\lambda_1, \lambda_2)$, see \eqref{eq:bulles-V}.
\begin{lemma}
  \label{lem:bulles-en-app}
  There exist constants $\lambda_0, \eta, R_0, c > 0$ such that if $\lambda = \frac{\lambda_1}{\lambda_2} \leq \lambda_0$, $|a_2| \leq \eta$ and $R \geq R_0$, then
  \begin{equation*}
    E(\bs U(\lambda_1, \lambda_2, a_2)) \geq 2E(\bs W) + c\lambda^\frac{N-2}{2}.
  \end{equation*}
\end{lemma}
\begin{proof}
  Without loss of generality we can assume that $\lambda_2 = 1$, $\lambda_1 = \lambda$ (it suffices to rescale).
  The conclusion follows from \cite[Lemma 2.7]{moi15p-2} applied for $\bs u^* = -\bs U^{a_2}$ (the proof given there is valid for $N \geq 3$).
\end{proof}
\begin{remark}
  \label{rem:bulles-coer}
  In Lemma~\ref{lem:bulles-coer} the fact that the bubbles have opposite signs has no importance, but it is crucial in Lemma~\ref{lem:bulles-en-app}.
  
\end{remark}

\subsection{Conclusion of the proof}
\begin{proof}[Proof of Theorem~\ref{thm:deux-bulles}]
  Suppose by contradiction that $\bs u(t): [0, T_+) \to \cE$ is a solution of \eqref{eq:nlw} such that \eqref{eq:deux-bulles} holds.
    Formula \eqref{eq:energy} and Lemma~\ref{lem:taylor} imply
  \begin{equation}
    \label{eq:bulles-taylor}
    \begin{aligned}
    2E(\bs W) = E(\bs U(\lambda_1, \lambda_2, a_2) + \bs g) &= E(\bs U(\lambda_1, \lambda_2, a_2)) + \la \vD E(\bs U(\lambda_1, \lambda_2, a_2)), \bs g\ra \\
    &+ \frac 12\la \vD^2 E(\bs U(\lambda_1, \lambda_2, a_2))\bs g, \bs g\ra + o(\|\bs g\|_\cE^2).
  \end{aligned}
  \end{equation}
  \paragraph{\textbf{Step 1 -- Coercivity.}}
  We will prove that for all $t$ there holds
\begin{equation}
  \label{eq:final-0}
  2a_2(t)b_2(t) +2\big(a_1^-(t)^2 + a_1^+(t)^2 + b_2(t)^2\big) \gtrsim \cN(\bs g(t), \lambda(t))^2
\end{equation}
(the functions $a_1^+$, $a_1^-$ and $b_2$ are defined in Proposition~\ref{prop:bulles-dtb}).

From \eqref{eq:beta-alpha-1} we have $|b_2(t)^2 - \la \alpha_{\lambda_2(t)}^+, \bs g(t)\ra^2| \lesssim |a_2|\cdot \|\bs g\|^2$.
Since $\la \alpha_{\lambda_2(t)}^-, \bs g(t)\ra = 0$, Lemma~\ref{lem:bulles-coer} and Lemma~\ref{lem:bulles-en-app} yield
\begin{equation*}
  \begin{aligned}
  E\big(\bs U(\lambda_1(t), \lambda_2(t), a_2(t))\big) -2E(\bs W) &+ \frac 12\big\la \vD^2 E\big(\bs U(\lambda_1(t), \lambda_2(t), a_2(t))\big)\bs g(t), \bs g(t)\big\ra \\
  &+2\big(a_1^-(t)^2 + a_1^+(t)^2 + b_2(t)^2\big)\geq c\cdot\cN(\bs g(t), \lambda(t))^2,
\end{aligned}
\end{equation*}
for $R \geq R_0$, with a constant $c > 0$ independent of $R$.

Recall that $2a_2(t) b_2(t) = -\la \vD E(\bs U_{\lambda_2}^{a_2}, \bs g\ra$. In view of \eqref{eq:bulles-taylor},
in order to prove \eqref{eq:final-0} it suffices to verify that
\begin{equation}
  \label{eq:final-1}
  |\la \vD E(\bs U(\lambda_1, \lambda_2, a_2)) - \vD E(\bs U_{\lambda_2}^{a_2}), \bs g\ra| \leq \frac c2\cdot \cN(\bs g, \lambda)^2
\end{equation}
provided that $R$ is large enough.
Without loss of generality we can assume that $\lambda_2 = 1$ and $\lambda_1 = \lambda$.
First we show that
\begin{equation}
  \label{eq:ustar-final1}
  \big|\la \vD E(\bs U(\lambda, 1, a_2)), \bs g\ra + \la\vD E(\bs V(\lambda, 1)), \bs g\ra - \la\vD E(\bs U^{a_2}), \bs g\ra\big| \ll \cN(\bs g, \lambda)^2.
\end{equation}
This is equivalent to
\begin{equation*}
  \int|f(U^{a_2} - V(\lambda, 1)) + f(V(\lambda, 1)) - f(U^{a_2})|\cdot|g|\ud x \ll \cN(\bs g, \lambda)^2.
\end{equation*}
By H\"older and Sobolev inequalities, it suffices to check that
\begin{equation}
  \label{eq:ustar-final0}
  \|f(-V(\lambda, 1) + U^{a_2}) + f(V(\lambda, 1)) - f(U^{a_2})\|_{L^{\frac{2N}{N+2}}} \ll \lambda^{\frac{N-2}{4}},
\end{equation}
which follows from the inequality
$$
||f(-V(\lambda, 1) + U^{a_2}) + f(V(\lambda, 1)) - f(U^{a_2})| \lesssim f'(W_\lambda) + 1.
$$

Next, we prove that if $R$ is large enough, then
\begin{equation}
  \label{eq:DEV-H-1}
  \|\vD E(\bs V(\lambda, 1))\|_{\cE^*} \leq \frac c4 \cdot \lambda^\frac{N-2}{4}.
\end{equation}
From \eqref{eq:db-V-W}, if $R$ is large then
\begin{equation}
  \label{eq:bulles-pointwise-Hm1-1}
  \|\Delta(W_\lambda - V(\lambda, 1))\|_{\dot H^{-1}} \lesssim \frac c8\cdot \lambda^\frac{N-2}{4}.
\end{equation}
We will prove that
\begin{equation}
  \label{eq:bulles-pointwise-Hm1-2}
\|f(W_\lambda) - f(V(\lambda, 1))\|_{L^\frac{2N}{N+2}} \ll \lambda^\frac{N-2}{4}.
\end{equation}
In the region $|x| \geq R\sqrt\lambda$ we have $V(\lambda, 1) = 0$ and
$$
\|f(W_\lambda)\|_{L^\frac{2N}{N+2}(|x| \geq R\sqrt\lambda)} = \|f(W)\|_{L^\frac{2N}{N+2}(|x| \geq R/\sqrt\lambda)} \sim \lambda^\frac{N+2}{4} \ll \lambda^\frac{N-2}{2}.
$$
In the region $|x| \leq R\sqrt\lambda$ we use the pointwise bound $|f(W_\lambda) - f(V(\lambda, 1))| \lesssim f'(W_\lambda)\cdot |W_\lambda - V(\lambda, 1)|$, the fact that $W_\lambda - V(\lambda, 1)$ is bounded in $L^\infty$ and the bound
\begin{equation}
  \label{eq:fpW-H-1}
\|f'(W_\lambda)\|_{L^\frac{2N}{N+2}(|x|\leq R\sqrt\lambda)} \ll \lambda^\frac{N-2}{4}.
\end{equation}
%
Now \eqref{eq:DEV-H-1} follows from \eqref{eq:bulles-pointwise-Hm1-1}, \eqref{eq:bulles-pointwise-Hm1-2} and $\Delta W_\lambda + f(W_\lambda) = 0$.

Estimate \eqref{eq:final-1} follows from \eqref{eq:ustar-final1} and \eqref{eq:DEV-H-1}.

\paragraph{\textbf{Step 2 -- Differential inequalities.}}
Observe that
\begin{equation}
  \label{eq:lambda12-diverge}
  \int_{0}^{T_+}\frac{1}{\lambda_1(t)}\ud t = \int_{0}^{T_+}\frac{1}{\lambda_2(t)}\ud t = +\infty.
\end{equation}
The proof is the same as the proof of \eqref{eq:lambda-diverge}.

For $m \in \bN$, $m \geq m_0$, let $t = t_m$ be the last time such that $\cN(\bs g(t), \lambda(t)) = 2^{-m}$. By continuity, $t_m$ is well defined if $m_0$ is large enough.

By Proposition~\ref{prop:bulles-dtb}, there exists a constant $C_1$ such that
  \begin{equation}
    \label{eq:bulles-a-destab}
    |a_1^+(t)| \geq C_1 \cdot \cN(\bs g(t), \lambda(t)) \quad\Rightarrow\quad \dd t|a_1^+(t)| \geq \frac{\nu}{2\lambda_1(t)}|a_1^+(t)|,\qquad \forall t\in [0, T_+).
  \end{equation}
Suppose that $|a_1^+(t_m)| \geq 2C_1 \cdot\cN(\bs g(t_m), \lambda(t_m))$. Since, by the definition of $t_m$, $\cN(\bs g(t), \lambda(t)) \leq \cN(\bs g(t_m), \lambda(t_m))$ for $t \geq t_m$, a simple continuity argument yields $|a_1^+(t_m)| \geq 2C_1 \cdot\cN(\bs g(t), \lambda(t))$
for all $t \geq t_m$. By \eqref{eq:bulles-a-destab} and \eqref{eq:lambda12-diverge}, this implies $|a_1^+(t)| \to +\infty$ as $t \to T_+$,
which is absurd. The same reasoning applies to $b(t)$, hence we get
\begin{equation}
  \label{eq:final-2}
  |a_1^+(t_m)| \lesssim \cN(\bs g(t_m), \lambda(t_m))^2,\qquad |b(t_m)| \lesssim \cN(\bs g(t_m), \lambda(t_m))^2.
\end{equation}
Thus \eqref{eq:final-0} forces
\begin{equation}
  \label{eq:final-3}
  |a_1^-(t_m)| \gtrsim \cN(\bs g(t_m), \lambda(t_m)) \gg \cN(\bs g(t_m), \lambda(t_m))^2.
\end{equation}
Consider the evolution on the time interval $[t_{m-1}, t_m]$. By definition of $t_{m-1}$ and $t_m$ for $t \in [t_{m-1}, t_m]$ there holds
  $\cN(\bs g(t), \lambda(t)) \leq 2\cdot \cN(\bs g(t_m), \lambda(t_m))$,
hence \eqref{eq:final-3} and Proposition~\ref{prop:bulles-dtb} allow to conclude that
\begin{equation*}
  \dd t |a_1^-(t)| \leq -\frac{\nu}{2\lambda_1(t)}|a_1^-(t)|,\qquad \forall t\in [t_{m-1}, t_m].
\end{equation*}
Since this holds for all $m$ sufficiently large, we deduce that there exists $t_0 < T_+$ such that
\begin{equation*}
  |a_1^-(t)| \leq |a_1^-(t_0)|\cdot \exp\Big(-\int_{t_0}^t\frac{\nu\ud \tau}{2\lambda_1(\tau)}\Big),\qquad \forall t\geq t_0.
\end{equation*}
Let $t \in [t_{m-1}, t_m]$. At time $t_{m}$ all the terms of \eqref{eq:final-0} except for the term $2a_1^-(t)^2$ are absorbed by the right hand side, hence $\cN(\bs g(t_{m}), \lambda(t_{m})) \lesssim |a_1^-(t_{m})|$. Using the definition of $t_{m-1}$ we obtain
$$
\begin{aligned}
\cN(\bs g(t), \lambda(t)) \leq 2\cN(\bs g(t_m), \lambda(t_m)) \lesssim |a_1^-(t_m)|&\lesssim |a_1^-(t_0)|\cdot \exp\Big(-\int_{t_0}^{t_m}\frac{\nu\ud \tau}{2\lambda_1(\tau)}\Big) \\ &\lesssim |a_1^-(t_0)|\cdot \exp\Big(-\int_{t_0}^t\frac{\nu\ud \tau}{2\lambda_1(\tau)}\Big).
\end{aligned}
$$
By \eqref{eq:bulles-mod}, this implies
\begin{equation*}
  |\lambda_1'(t)| + |\lambda_2'(t)| \lesssim \exp\Big(-\int_{t_0}^t\frac{\nu\ud \tau}{2\lambda_1(\tau)}\Big),\qquad \forall t\geq t_0.
\end{equation*}
Dividing both sides by $\lambda_1(t)$ and integrating we obtain that $\log \lambda_1(t)$ converges as $t \to T_+$.
Dividing both sides by $\lambda_2(t)$, using the fact that $\lambda_2(t) \geq \lambda_1(t)$ for $t \geq t_0$ and integrating we obtain that $\log\lambda_2(t)$ converges as $t \to T_+$.
Hence $\log \lambda(t)$ converges, which is impossible.
\end{proof}

\begin{remark}
  \label{rem:necessaire}
  An analogous proof using the linear stability and instability components $\alpha_{\lambda_2}^+$ and $\alpha_{\lambda_2}^-$
  instead of the refined modulation and instability component $\beta_{\lambda_2}^{a_2}$
  would yield $\lambda_2(0) \to \lambda_0 \in (0, +\infty)$ (hence $T_+ = +\infty$)
  and $|\log \lambda_1(t)| \gtrsim t$ as $t \to +\infty$, but would not (at least directly) lead to a contradiction.
\end{remark}

\appendix

\section{Elementary lemmas}
\begin{lemma}
  \label{lem:hartman-1}
  Let $\psi :\bR \to \bR$ be an analytic function such that $\psi(0) = 0$ and $\psi'(0) \neq 0$.
  Then there exists a local analytic diffeomorphism $y = \varphi(x)$ near $x = 0$ such that $\varphi(0) = 0$, $\varphi'(0) = 1$ and
  \begin{equation}
    \label{eq:hartman-1}
    \varphi'(x)\cdot \psi(x) = \varphi(x)\cdot \psi'(0).
  \end{equation}
\end{lemma}
\begin{remark}
  \label{rem:hartman-1}
  Equation \eqref{eq:hartman-1} expresses the fact that the change of variable $y = \varphi(x)$
  transforms the differential equation $\dot x = \psi(x)$ into $\dot y = \psi'(0)y$.
\end{remark}
\begin{proof}
  Without loss of generality we can assume that $\psi'(0) = 1$.
  We set:
  \begin{equation*}
    \varphi(x) := \psi(x)\cdot \exp\Big(\int_0^x \frac{1-\psi'(z)}{\psi(z)}\ud z\Big)
  \end{equation*}
  and it suffices to verify that $\varphi$ has the required properties.
\end{proof}

Recall that we denote $f(u) := |u|^\frac{4}{N-2}u$ and $R(v) := f(W+v) - f(W) - f'(W)v$. Notice that $f'$~is not Lipshitz for $N > 6$.
\begin{lemma}
  \label{lem:nonlin-anal}
  The mapping $\cR$ is analytic from $B_{Y^k}(0, \eta)$ to itself if $k \geq k_0$ and $\eta$ is small. Its derivative is given by
  \begin{equation*}
    \scrL(Y^k) \owns \vD_{v} \cR = \big(h \mapsto (f'(W+v) - f'(W))h\big).
  \end{equation*}

  The same conclusion holds if we replace $Y^k$ by $BC_{\wt \nu}$ for $\wt\nu \geq 0$.
\end{lemma}
\begin{proof}
  We have an isomorphism
  \begin{equation*}
    \Phi: Y^k \to H^k,\qquad \Phi(v) := (1+|x|^k)v,
  \end{equation*}
  so it suffices to show that $\Phi\circ \cR\circ\Phi^{-1}$ is analytic from $B_{H^k}(0,\eta)$ to itself. Let $w \in B_{H^k}(0, \eta)$.

  Let $f(1+z) = |1+z|^\frac 4N(1+z) = \sum_{n=0}^{+\infty}a_n z^n$. The series converges for $|z| < 1$.
  We have a series expansion:
  \begin{equation*}
    \cR(\Phi^{-1}w) = \sum_{n=2}^{+\infty}a_n W^{\frac{N+2}{N-2} - n}\frac{w^n}{(1+|x|^k)^n} = \frac{1}{1+|x|^k}\frac{W^{\frac{6-N}{N-2}}}{1+|x|^k}\sum_{n=2}^{+\infty}a_n\Big(\frac{1}{W\cdot(1+|x|^k)}\Big)^{n-2}w^n.
  \end{equation*}
  We see that $\frac{W^\frac{6-N}{N-2}}{1+|x|^k} \in H^k$ if $k$ is large enough and that the last series converges strongly in $H^k$ if $\eta$ is small.

  In the case of the space $BC_{\wt\nu}$ the proof is the same.
\end{proof}
\begin{lemma}
  \label{lem:weak-linea}
  There exists $k = k(N) \in \bN$ and $\eta = \eta(N) > 0$ such that if $\psi \in Y^k$ and $|a| \leq \eta$, then for all $g \in \dot H^1$ such that $\|g\|_{\dot H^1} \leq \eta$ there holds
  \begin{align}
    |\la \psi, f(U^a + g) - f(U^a) - f'(U^a)g\ra| &\lesssim \|g\|_{\dot H^1}^2, \label{eq:weak-linea-1} \\
    |\la \psi, \big(f'(U^a) - f'(W)\big)g\ra| &\lesssim |a|\cdot \|g\|_{\dot H^1}, \label{eq:weak-linea-2}
  \end{align}
  with a constant depending on $\psi$.
\end{lemma}
\begin{proof}
  For $N \in \{3, 4, 5\}$ this follows directly from the Sobolev and H\"older inequalities (even for $\psi \in \dot H^1$).

  For $N \geq 6$ we use the pointwise bound
  $$
  |f(U^a + g) - f(U^a) - f'(U^a)g| \lesssim |f''(U^a)|\cdot |g|^2.
  $$
Here, $f''$ is a negative power. Since $U^a$ has slow decay, $\psi\cdot |f''(U^a)| \in L^\frac N2$ if $\psi \in Y^k$ and $k$ is large enough.
The conclusion follows from the H\"older inequlity.

The proof of \eqref{eq:weak-linea-2} is similar.
\end{proof}
\begin{lemma}
  \label{lem:taylor}
  Let $\gamma := \min\big(3, \frac{2N}{N-2}\big)$. For any $M > 0$ there exists $C > 0$ and $\eta > 0$ such that if $\|\bs v\|_\cE \leq M$ and $\|\bs g\|_\cE \leq \eta$, then
  \begin{equation*}
  \big|E(\bs v + \bs g) - E(\bs v) - \la \vD E(\bs v), \bs g\ra - \frac 12\la \vD^2 E(\bs v)\bs g, \bs g\ra\big| \leq C\|\bs g\|_\cE^\gamma.
  \end{equation*}
\end{lemma}
\begin{proof}
  In dimension $N \in \{3, 4, 5\}$ this follows from the pointwise inequality
  \begin{equation}
    \label{eq:pointwise-1}
  \big|F(k+l) - F(k) - f(k)l - \frac 12 f'(k)l^2\big| \lesssim |f''(k)|\cdot|l^3| + |F(l)|,\qquad k, l\in \bR,
  \end{equation}
  whereas for $N \geq 6$ from
  \begin{equation}
    \label{eq:pointwise-2}
  \big|F(k+l) - F(k) - f(k)l - \frac 12 f'(k)l^2\big| \lesssim |F(l)|,\qquad k, l\in \bR.
  \end{equation}
  In order to prove bounds \eqref{eq:pointwise-1} and \eqref{eq:pointwise-2}, notice that they are homogeneous and invariant by changing signs of both $k$ and $l$,
  hence it can be assumed that $k = 1$ (for $k = 0$ the inequalities are obvious). Now for $|l| \leq \frac 12$ the conclusion follows from the asymptotic expansion of $F(1 + l)$
  and for $|l| \geq \frac 12$ the bounds are evident.
\end{proof}

\bibliographystyle{plain}
\bibliography{coercivite}

\end{document}